\documentclass{article}
\usepackage{authblk}
\usepackage[T1]{fontenc}
\usepackage{lmodern}
\usepackage{amsfonts,amsmath,amssymb,amsthm}
\usepackage{graphicx,color}
\usepackage[margin=1in]{geometry}
\usepackage{tikz-cd}
\usepackage{tikz}
\usepackage{blkarray}
\usepackage{enumitem}
\usepackage{caption}
\usepackage{subcaption}
\usepackage{cite}
\usepackage[dvipsnames]{xcolor}
\usepackage{comment}
\usepackage[colorlinks=true,linkcolor=RoyalBlue,citecolor=PineGreen,urlcolor=RoyalBlue]{hyperref}
\usepackage{cleveref}
\usepackage{thmtools} 
\usepackage{thm-restate}
\usepackage{kbordermatrix}

\newtheorem{theorem}{Theorem}[section]
\newtheorem{proposition}[theorem]{Proposition}
\newtheorem{lemma}[theorem]{Lemma}
\newtheorem{corollary}[theorem]{Corollary}
\newtheorem{conjecture}[theorem]{Conjecture}

\theoremstyle{definition}
\newtheorem{definition}[theorem]{Definition}
\newtheorem{example}[theorem]{Example}
\theoremstyle{remark}
\newtheorem{remark}[theorem]{Remark}

\newcommand{\Stab}{\textup{Stab}} 
\newcommand{\im}{\textup{im}} 
\newcommand{\codim}{\textup{codim}} 

\title{Homological methods in rigidity theory using graphs of groups}
\author{Joannes Vermant\thanks{Ume{\aa} Universitet, Universitetstorget 4, 901 87 Umeå, Sweden}}
\date{}
\begin{document}

\maketitle

\begin{abstract}
In recent work, Stokes and Vermant considered graph-of-groups realisations of hypergraphs as a new description of rigidity-theoretic problems. In this paper, we show that the infinitesimal aspects of graph-of-groups realisations can be analysed using cellular sheaves and their cohomology. Using these tools, we give an algebraic condition for Henneberg moves to preserve independence, and we prove that the infinitesimal rigidity and flexibility of certain graph-of-groups realisations are generic properties. We use these results to show that whenever a rigidity-theoretic problem is defined in a Lie group $G$ using a $1$-dimensional connected subgroup $H$ with $N_{G}(H)/H$ finite, then the so-called Maxwell-count leads to a necessary and sufficient condition for minimal rigidity, generalising various known results in the literature.
\end{abstract}

\section{Introduction}\label{sec: Introduction}
In structural rigidity theory, one studies bar-joint frameworks, which are graphs embedded in Euclidean space, such that each edge is considered as a rigid bar and each vertex is considered as a universal joint. A bar-joint framework is said to be rigid if there are no non-trivial ways to deform the framework without breaking it. For almost all choices of coordinates, it turns out that rigidity is a property that depends only on the underlying graph \cite{cf04a725-5f21-3920-aafd-d5f8dd4a83ae}, and one can hence speak of rigidity as a graph-theoretical property. 

In dimension $1$, rigidity is characterised by connectivity, and in dimension $2$, a combinatorial characterisation of rigidity is given by the Geiringer-Laman theorem \cite{PollaczekGeiringer1927, Laman1970}. In dimension $3$ and higher, characterising rigidity is an open problem. The characterisation in dimension $2$ describes rigidity in terms of a so-called sparsity matroid \cite[Section 13.5]{Frank2012}, allowing for algorithms that check rigidity in polynomial time. Additionally, researchers have defined and studied the rigidity of many other combinatorial-geometric structures, proving that they are characterised by sparsity conditions \cite{ Whiteley1989, TAY198495, Tay1989, Katoh2011, 10.1093/imrn/rny170}. 

In this paper, we further develop the theory of graph-of-groups realisations, building upon the article by Stokes and the present author \cite{graphsofgroups}. Graph-of-groups realisations provide a general framework for rigidity theory, which describes rigidity using only (Lie) group theory. One can describe the position of an articulated structure by only keeping track of the stabiliser subgroup of each geometric element. One can then describe the motions of the structure using group elements and the given subgroups. In \cite{graphsofgroups}, the approach is shown to model the classical theory of bar-joint frameworks; it also models, for example, scenes and parallel redrawings \cite{Whiteley1989} and projective rigidity \cite{https://doi.org/10.48550/arxiv.2407.17836}. 

For all graph-of-groups realisations, one can derive a counting condition similar to the one in the Geiringer-Laman theorem, which provides a necessary condition for rigidity. In this article, we make progress towards answering the question of when this condition is also sufficient. To do this, one needs to first define \textit{generic rigidity}, a problem that was left unresolved in \cite{graphsofgroups}. We prove a genericity result for a broad class of realisations in \Cref{thm: Generic}.
The main theorem states that for sufficiently generic graph-of-groups realisations, where the stabilisers have dimension $1$, the rigidity is characterised by the necessary condition. See \Cref{thm: main theorem}, and \Cref{cor: main - thm} for a precise statement. 

Similar statements have been conjectured for another general rigidity model, called $g$-rigidity \cite[Conjecture 4.13]{cruickshank2023identifiability}.  \Cref{thm: main theorem} can also be seen as being similar in spirit to the results in \cite{streinu2011natural}, providing a general way to give realisations of certain sparsity matroids. 

We discuss briefly the applications of this theorem. We show that the Geiringer-Laman theorem can be deduced from the main theorem, and one can also immediately deduce similar results for spherical and hyperbolic rigidity, a fact usually explained through coning \cite{whiteley1983cones, schulze2012coning}. We also show that \Cref{thm: main theorem} applies to parallel rigidity, which was previously characterised in \cite{Whiteley1996, develin2007rigidity}.

The proof of \Cref{thm: main theorem} is an inductive proof. We use a result by Frank and Szegö \cite{frank2003constructive} that provides an inductive characterisation of sparse graphs. Szegö \cite{SZEGO20061211} conjectured that similar results should hold for other values. We use novel methods to show that the extensions preserve independence after suitably defining independence in the graph-of-groups setting. We remark that understanding $2$-extensions for $3$-dimensional rigidity is an open problem, the resolution of which would represent significant progress towards $3$-dimensional rigidity \cite{Cruickshank2014, 2022}.

The main tool that we use is that the infinitesimal motions can be interpreted as the $0$th cohomology group of cellular sheaves on graphs. Cellular sheaves \cite{currysheaves, Hansen2019} are variants of classical sheaves, and they can be used to study and address local-to-global type questions. 

Modelling the infinitesimal motions as a sheaf-cohomology group, we shall use long exact sequences to study the relation of infinitesimal motions to other vector spaces. In fact, it is through the use of a long exact sequence that we can prove the inductive constructions preserve rigidity for a general class of sheaves that we call \textit{motion sheaves}. Naturally defined classes of motion sheaves carry the structure of a real algebraic variety, being a product of Grassmannians, and we will consider how the dimensions of these cohomology groups can vary within the variety. We will establish that the dimensions of the cohomology groups exhibit upper semi-continuity with respect to the Zariski topology. For the inductive constructions, we need to consider a class of sheaves related to the motion sheaves that we call \textit{associated sheaves}, and their generic behaviour is characterised in \Cref{thm: main_thm_ass_sheaves}. These associated sheaves may be of independent interest.

To apply the results to graph-of-groups realisations, one needs to make a geometric analysis of which motion sheaves arise from graph-of-groups realisations. Within the mentioned varieties, the sheaves coming from graph-of-groups realisations will cut out semi-algebraic subsets. In the case where stabilisers are $1$-dimensional, and in the case of parallel redrawings, these semi-algebraic subsets are full-dimensional, meaning that in these cases, the graph-of-groups will behave like generic motion sheaves. The main theorem, \Cref{thm: main theorem}, establishes the generic behaviour of motion sheaves in these cases. Classes of motion sheaves that come from graph-of-groups do not need to exhibit generic behaviour, and this is, for example, the case for $3$-dimensional rigidity.

The duals of cellular sheaves, cosheaves, have been exploited to explain graphic statics from a homological perspective. This can be read in the PhD thesis of Cooperband \cite{cooperband2024cellular}, and in the work by Cooperband, Ghrist, Hansen, Lopez, McRobie, Millar, and Schulze upon which the thesis expands \cite{cooperband2023towards, cooperband2024cellular,cooperband2024equivariant,cooperband2024homology,cooperband2025unified}. The treatment of graphic statics using these techniques simplifies and clarifies many constructions and also generalises them cleanly to new contexts. However, homological methods had already been applied to the study of rigidity and liftings of scenes by Crapo, Whiteley, and Tay \cite{Crapo-projective-configurations, crapo-geometric-homology, Whiteley1998-geometric-homology, Crapo2004StressesI3, TAY2000102} (see also Part III of the survey article \cite[part III]{Whiteley1996}). The model utilised in the current work is most similar to the one used in \cite{cooperband2025unified}; however, the type of questions we consider is quite different, and the model is more general. Much of the previous work has also focused on higher homology groups, which were associated with structures such as simplicial complexes or graph embeddings on surfaces. In the current work, we only consider sheaves on graphs, though similar higher-rank structures are a natural and promising avenue for further work.

\section{Motion sheaves}\label{sec: motion sheaves}

\subsection{Background on cellular sheaves and their cohomology}\label{sec: Background - sheaves}
We briefly give the necessary background on cellular sheaves. See \cite{currysheaves, Hansen2019} for a more comprehensive treatment of cellular sheaves. We will only be interested in sheaves on graphs. 
\begin{definition}
A sheaf $\mathcal{F}$ on a graph $\Gamma=(V,E)$ consists of a finite dimensional vector space $\mathcal{F}(v)$ for all $v\in V$, as well as a finite dimensional vector space $\mathcal{F}(e)$ for all $e\in E$, together with linear maps (called restriction maps)
\begin{align*}
r^{v}_{e}: \mathcal{F}(v) &\rightarrow \mathcal{F}(e)
\end{align*}
for each vertex-edge inclusion $v\in e$. We will denote the class of all sheaves on a given graph by $\textup{Sh}(\Gamma)$. 
\end{definition}
One defines 
\begin{align*}
    \mathcal{F}(E) &= \bigoplus_{e\in E} \mathcal{F}(e),\\
    \mathcal{F}(V) &= \bigoplus_{v\in V} \mathcal{F}(v),
\end{align*}
and the map
\begin{equation}
d : \mathcal{F}(V) \rightarrow \mathcal{F}(E): (w_{v_1}, \cdots, w_{v_n}) \mapsto (\sum_{v \sim e} [v:e]r^{v}_{e}(w_v))_{e\in E},
\end{equation}
where $[v:e] =\pm 1$, oriented in such a way that for every edge $e=vu$, one has $[v:e]= -[u:e]$.

One then defines the cohomology groups
\begin{align*}
    H^{0}(\Gamma, \mathcal{F}) &= \ker(d)\\
    H^{1}(\Gamma, \mathcal{F}) &= \text{coker}(d) = \frac{\mathcal{F}(E)}{\textup{im}(d)}.
\end{align*}
It is clear that this is independent of the given orientation on $\Gamma$. Elements of $H^{0}(\Gamma, \mathcal{F})$ are called sections.

Let $\mathcal{F}, \mathcal{G}$ be sheaves on a graph $\Gamma$. A morphism of sheaves $\mathcal{F}\rightarrow \mathcal{G}$ consists of a collection of $\vert V\vert +\vert E\vert$ linear maps:
\begin{align*}
   \alpha_v:& \mathcal{F}(v)\rightarrow \mathcal{G}(v)\\
   \alpha_e:& \mathcal{F}(e)\rightarrow \mathcal{G}(e),
\end{align*}

such that for any vertex-edge inclusion, the following diagram commutes:
\begin{center}
\begin{tikzcd}
\mathcal{F}(v) \arrow[r, "\alpha_v"] \arrow[d, "r^{v}_e"] & \mathcal{G}(v) \arrow[d, "r^{v}_e"] \\
\mathcal{F}(e) \arrow[r, "\alpha_e"]                      & \mathcal{G}(e)                     
\end{tikzcd}
\end{center}

A morphism of sheaves is said to be injective and surjective, respectively, if all the $\alpha_v$ and $\alpha_e$ are. Given a morphism of sheaves $\alpha: \mathcal{F}\rightarrow \mathcal{G}$, one can also define the sheaves $\ker(\alpha)$ and $\text{im}(\alpha)$, which are the kernel and image of a morphism of sheaves. They are defined by:
\begin{align*}
    \ker(\alpha)(x) &:= \ker(\alpha_x) \text{ for all  } x\in V\cup E,\\
    \im(\alpha)(x) &:= \im(\alpha_x) \text{ for all  } x\in V\cup E,
\end{align*}
and $r^{v}_e$ is given by the restriction map in $\mathcal{F}$ or $\mathcal{G}$ respectively. One has that $\alpha: \mathcal{F} \rightarrow \mathcal{G}$ is surjective if $\im(\alpha) = \mathcal{G}$.

A map of cellular sheaves induces a map of sections $\alpha_{\Gamma}: H^{0}(\Gamma,\mathcal{F}) \rightarrow H^{0}(\Gamma,\mathcal{G}): \sum_{v\in V} w_v \mapsto \sum_{v\in V} \alpha_{v}(w_v)$ as well as maps $H^{1}(\Gamma,\mathcal{F}) \rightarrow H^{1}(\Gamma,\mathcal{G}): [\sum_{e\in E} w_{e}]\mapsto [\sum_{e\in E} \alpha_{e}(w_{e})]$. However, a key point about sheaves is that $\alpha_{\Gamma}$ need not be surjective if $\alpha$ is. As for modules, one says that a sequence of morphisms 
\begin{equation*}
    \mathcal{F} \xrightarrow{\alpha} \mathcal{G}\xrightarrow {\beta}\mathcal{H} 
\end{equation*}
is exact if $\ker(\beta)=\text{im}(\alpha)$.

One sheaf that will play an important role in what follows is the constant sheaf with values in a vector space $\mathbb{V}$.
\begin{definition}
    Let $\mathbb{V}$ be a vector space, and let $\Gamma=(V,E)$ be a graph. We denote by $\underline{\mathbb{V}}$ the sheaf with $\underline{\mathbb{V}}(x) = \mathbb{V}$ for all $x\in V\cup E$, and the restriction maps are all given by the identity. 
\end{definition}

We will use the following fact throughout the article. Any exact sequence of sheaves on a graph $\Gamma:$
\begin{equation*}
    0\rightarrow \mathcal{F} \rightarrow \mathcal{G}\rightarrow \mathcal{H} \rightarrow 0
\end{equation*}
gives rise to the following exact sequence in cohomology:
\begin{equation*}
    0 \rightarrow H^0(\Gamma, \mathcal{F}) \rightarrow H^0(\Gamma, \mathcal{G})\rightarrow H^0(\Gamma, \mathcal{H})\rightarrow H^1(\Gamma, \mathcal{F})\rightarrow H^1(\Gamma, \mathcal{G})\rightarrow H^1(\Gamma, \mathcal{H})\rightarrow 0.
\end{equation*} 
The terms $H^{k}(\Gamma, \cdot)$ are $0$ for $k\geq 2$ since $\Gamma$ is a graph.
\subsubsection{Restricting to a sub-graph}
One always has naturally defined morphisms of sheaves whenever one has a subgraph $\Gamma' \subseteq \Gamma$. Let $\mathcal{F}$ be a sheaf on a graph $\Gamma$. For any subgraph $\Gamma'\subset \Gamma$ (not necessarily induced), we can consider $\mathcal{F}^{(1)}_{\vert \Gamma'}$ as a sheaf on $\Gamma'$, which is defined as
\begin{equation}
\mathcal{F}^{(1)}_{\vert \Gamma'}(x) = \mathcal{F}(x)
\end{equation}
with the restriction maps defined in the same way as in $\mathcal{F}$. Alternatively, one defines the restriction $\mathcal{F}^{(2)}_{\vert \Gamma'}$ as a sheaf on $\Gamma$, by setting
\begin{align*}
\mathcal{F}_{\vert \Gamma'}^{(2)}(x) &= \mathcal{F}(x) \text{ if $x\in V(\Gamma') \cup  E(\Gamma').$}
\\
\mathcal{F}_{\vert \Gamma'}^{(2)}(x) &= 0 \text{ otherwise}.
\end{align*}
The map $\mathcal{F}^{(2)}_{\vert \Gamma'}(v) \rightarrow \mathcal{F}^{(2)}_{\vert \Gamma'}(e)$ is given by the map coming from $\mathcal{F}$, whenever $e\in E(\Gamma')$, and the zero map otherwise. One may check that 
\begin{align*}
H^{0}(\Gamma', \mathcal{F}^{(1)}_{\vert \Gamma'}) &\cong H^{0}(\Gamma, \mathcal{F}^{(2)}_{\vert \Gamma'}). \\
H^{1}(\Gamma', \mathcal{F}^{(1)}_{\vert \Gamma'}) &\cong H^{1}(\Gamma, \mathcal{F}^{(2)}_{\vert \Gamma'}).
\end{align*}  

One has a surjective map $\alpha: \mathcal{F}\rightarrow \mathcal{F}_{\vert \Gamma}^{(2)}$, given by $\alpha_{x} = \textup{id}$ if $x\in \Gamma'$, and $\alpha_{x} = 0$ if $x\notin \Gamma'$. One thus also has a short exact sequence:
\begin{equation}\label{eq: les_restriction}
0 \rightarrow \ker(\alpha) \rightarrow \mathcal{F}\rightarrow \mathcal{F}_{\vert \Gamma}^{(2)} \rightarrow 0.
\end{equation}
hence, one has, by the long exact sequence in cohomology (since $H^{2}=0$ for graphs):
\begin{lemma} \label{lem: Stresses extend}
The map $H^{1}(\Gamma, \mathcal{F}) \rightarrow H^{1}(\Gamma, \mathcal{F}\vert_{\Gamma'}^{(2)})$ induced by \eqref{eq: les_restriction} is surjective. In particular, if $H^{1}(\Gamma, \mathcal{F})= 0$, then $ H^{1}(\Gamma, \mathcal{F}\vert_{\Gamma'}^{(2)})=0$.
\end{lemma}

We will use restrictions of sheaves $\mathcal{F}_{\Gamma'}$ throughout the paper. Since we will only be interested in the cohomology groups, we will not specify whether $\mathcal{F}_{\Gamma'}$ is a sheaf on $\Gamma'$ or on $\Gamma$.

\subsection{Definitions and statement of main theorem.}

The motivation for the following definition is that one can describe the infinitesimal motions of graph-of-groups realisations as a $0$th cohomology group of a sheaf on the incidence graph, as we will see in \Cref{sec: applications to gogs}. We recall that the incidence graph of a hypergraph $\Gamma=(V, E)$ is the graph with vertex set $X$, where $X= V\cup E$, and with edge set $I = \{\{v,e\} \subseteq V\cup E ~ |~ v \in e \}$. This graph is denoted by $I(\Gamma) = (X, I)$, and we denote the edges of this graph by $v\sim e$ instead of $ve$. In this article, we denote that $S$ is a vector subspace of a vector space $\mathbb{V}$ by $S\leq \mathbb{V}$.

\begin{definition}
    Let $\Gamma=(V, E)$ be a hypergraph, and let $\mathbb{V}$ be a finite dimensional real vector space. For any choices of subspaces $S_v$ for all $v\in V$, we can define a sheaf $\mathcal{S}$ on $I(\Gamma)=(X, I)$ by
    \begin{align*}
        \mathcal{S}(v) &= S_v \leq \mathbb{V} \text{ for } v \in X \text{ representing vertices of $\Gamma$},\\
        \mathcal{S}(e) &= \bigcap_{v\in e} \mathcal{S}(v) \leq \mathbb{V} \text{ for } e\in X \text{ representing edges of $\Gamma$},\\
        \mathcal{S}(e \sim v) &= \mathcal{S}(v).
    \end{align*}
    The maps 
    \begin{equation*}
        \begin{array}{cc}
            r^{v}_{v\sim e}:  \mathcal{S}(v) \rightarrow \mathcal{S}(v \sim e), & 
            r^{e}_{v\sim e}:\mathcal{S}(e) \rightarrow \mathcal{S}(v \sim e) 
        \end{array}
    \end{equation*}
    are given by the identity and inclusion. We call such sheaves subspace sheaves. We define a motion sheaf to be a sheaf of the form 
        \begin{align*}
        \mathcal{M}(v) &= \mathbb{V}/\mathcal{S}(v)  \text{ for } v \in V(I(\Gamma)) \text{ representing vertices of $\Gamma$},\\
        \mathcal{M}(e) &= \mathbb{V}/\mathcal{S}(e) \text{ for } e\in V(I(\Gamma)) \text{ representing edges of $\Gamma$},\\
        \mathcal{M}(e \sim v) &=  \mathbb{V}/\mathcal{S}(v).
    \end{align*}
    The maps 
    \begin{equation*}
        \begin{array}{cc}
            r^{v}_{v\sim e}: \mathcal{M}(v) \rightarrow \mathcal{M}(v \sim e) & 
            r^{e}_{v\sim e}: \mathcal{M}(e) \rightarrow \mathcal{M}(v \sim e) 
        \end{array}
    \end{equation*}
    are given by the identity and the reduction modulo $\mathcal{S}(v)$, which is well defined since $\mathcal{S}(e) \subseteq \mathcal{S}(v)$.
\end{definition}
Given a subspace sheaf $\mathcal{S}$ and a motion sheaf $\mathcal{M}$ defined using the same subspaces $S_{v}\leq \mathbb{V}$ for $v\in V$, one has that $\mathcal{M} = \underline{\mathbb{V}}/\mathcal{S}$, where $\underline{\mathbb{V}}$ is a constant sheaf. More precisely, one has a short exact sequence of sheaves
    \begin{equation}\label{ses}
        0\rightarrow \mathcal{S} \xrightarrow{i}   \underline{\mathbb{V}} \xrightarrow{p} \mathcal{M} \rightarrow 0.
    \end{equation}
    The map $i$ is defined using the inclusions $\mathcal{S}(x) \rightarrow \mathbb{V}$ for all $x\in X$ and $x\in I$, and $p$ is defined using the projection map $\mathbb{V} \rightarrow  \mathbb{V}/\mathcal{S}(x)$ for all $x\in X$ and $x\in I$. To somewhat ease the notation, we denote the induced map $p_{I(\Gamma)}: H^{0}(I(\Gamma), \underline{\mathbb{V}}) \rightarrow H^{0}(I(\Gamma), \mathcal{M})$ by $p_{\Gamma}$ instead.

The sets of subspace and motion sheaves can be given the structure of a real algebraic variety. We denote by $\textup{Gr}(s, \mathbb{V})$ the Grassmannian of $s$-dimensional subspaces $\mathbb{V}$, considered as a real algebraic variety. We will consider only the real points on real algebraic varieties, i.e., we do not consider complex points. Let $s,n\in \mathbb{N}$, with $s < n$ and let $\Gamma$ be a hypergraph. Fix a real vector space $\mathbb{V}$ of dimension $n$ and define
\begin{align*}   
S_{s, n}(\Gamma):=\{\mathcal{S} \in \textup{Sh}(I(\Gamma)) ~:~ &\mathcal{S}(v)\in \textup{Gr}(s,\mathbb{V}), ~ \mathcal{S}(e)  = \cap_{v\in e} \mathcal{S}(v) \\
&\mathcal{S}(v\sim e) = \mathcal{S}(v), \text{ with $r^{x}_{v\sim e}$ being the inclusion for $x=v$ or $x=e$}\}.
\end{align*}

We also introduce a notation for the set of all motion sheaves associated with sheaves in $S_{s,n}$:
$$
M_{s, n}(\Gamma):=\{\mathcal{M} \in \textup{Sh}(I(\Gamma))~:~ \mathcal{M}= \underline{\mathbb{V}}/\mathcal{S}, \mathcal{S} \in S_{s,n}  \},$$

We can naturally identify both $S_{s, n}$ and $M_{s,n}$ with the real algebraic variety $\textup{Gr}(s, \mathbb{V})^{|V|}$, and we equip the sets $S_{s,n}$ and $M_{s,n}$ with this structure. The necessary condition that can be derived is a sparsity condition on a hypergraph. See \cite{MR2552675} for more information on combinatorial aspects of such graphs. 

\begin{definition}\label{def: sparse}
    Let $d, \ell \in \mathbb{N}$ with $\ell \leq dr-1$. An $r$-uniform hypergraph $\Gamma=(V, E)$ is said to be $(d, \ell)$-sparse if for all subsets $V' \subseteq V$, with $|V'|\geq r$, one has $|E[V']| \leq d|V'| - \ell,$ where $E[V']$ is the set of hyperedges supported on $V'$. The hypergraph $\Gamma$ is said to be $(d, \ell)$-tight if it is $(d, \ell)$-sparse and additionally $|E|= d|V| - \ell$. 
\end{definition}

For any graph $\Gamma= (V, E)$ and $a \in \mathbb{N}$, we denote by $a \Gamma$ the multigraph with vertex set $V$ and edge set $\biguplus_{i=1}^{a} E$ (i.e., there are $a$ copies of each edge). We are now ready to state the main theorem. 
\begin{theorem}\label{thm: main theorem}
    Let $\Gamma$ be a graph, and $n\geq 3$. There is a nonempty Zariski open subset of $M_{1, n}(\Gamma)$ such that $H^{1}(I(\Gamma), \mathcal{M}) = 0$ if and only if $(n-2)\Gamma$ is $(n-1, n)$-sparse. 
\end{theorem}
In other words, this means that for generic points of $M_{1, n}(\Gamma)$, the corresponding motion sheaves will satisfy $H^{1}(I(\Gamma), \mathcal{M}) = 0$ if and only if $(n-2)\Gamma$ is $(n-1,n)$-sparse. In terms more familiar to rigidity theorists, $H^{1}(I(\Gamma), \mathcal{M})=0$ is the notion corresponding to independence.

\subsection{Basic observations}
\begin{theorem}\label{thm: Maxwell-rule}
Let $\Gamma=(V, E)$ be a hypergraph, and let $\mathcal{M}$ be a motion sheaf on $I(\Gamma)=(X, I)$. One has
\begin{equation*}
\dim(H^{0}(I(\Gamma), \mathcal{M})) - \dim(H^{1}(I(\Gamma),  \mathcal{M}))= \sum_{x\in X} \dim\left(\mathcal{M}(x)\right) - \sum_{v\sim e\in I} \dim(\mathcal{M}(v)).
\end{equation*}
\end{theorem} 
\begin{proof}
By the rank-nullity theorem, applied to the boundary operator, one has
\begin{equation*}
 \dim( \text{ker}(d)) + \dim( \text{im}(d))= \dim(\mathcal{M}(X)).
 \end{equation*}
 Hence, 
 \begin{equation*}
 \dim( \text{ker}(d)) - \dim( \text{coker}(d))= \dim(\mathcal{M}(X)) - \dim(\mathcal{M}(I)).
 \end{equation*}
Counting the dimensions of $\mathcal{M}(X)$ and $\mathcal{M}(I)$ yields the result. 
\end{proof}

\Cref{thm: Maxwell-rule} is a general version of the Maxwell rule and it indicates that $H^{1}(I(\Gamma), \mathcal{M})$ plays the same role as stresses in the classical setting. From the short exact sequence \eqref{ses}, one has the following long exact sequence:
\begin{equation}\label{les}
    \begin{array}{cl}
    0 \rightarrow & H^0(I(\Gamma), \mathcal{S}) \rightarrow H^0(I(\Gamma), \underline{\mathbb{V}}) \xrightarrow{p_{\Gamma}} H^0(I(\Gamma), \mathcal{M}) \\[10pt]
    &\xrightarrow{\delta} H^1(I(\Gamma), \mathcal{S}) \rightarrow H^1(I(\Gamma), \underline{\mathbb{V}}) \rightarrow H^1(I(\Gamma), \mathcal{M}) \rightarrow 0.
\end{array}
\end{equation}
We call elements $(W_{x})_{x\in V\cup E} = H^{0}(I(\Gamma), \mathcal{M})$ such that there exists a $W\in \mathbb{V}$ with
\begin{equation*}
    W_x = W \textup{ mod } \mathcal{S}(x) \textup{ for all } x\in V\cup E
\end{equation*}
trivial infinitesimal motions. An easy computation shows that 
\begin{align*}
    H^{0}(I(\Gamma), \mathcal{S}) & \cong \bigoplus_{ C \in K(\Gamma) } \left(\bigcap_{v \in V(C)} \mathcal{S}(v)\right),
\end{align*}
where $K(\Gamma)$ denotes the set of connected components of $\Gamma$. By a connected component, we mean a maximal vertex set $V'\subseteq V$ such that $I(\Gamma[V'])$ is connected. It is equally easy to verify that 
\begin{equation*}
   H^{0}(I(\Gamma), \underline{\mathbb{V}}) \cong \bigoplus_{ C \in K(\Gamma) } \mathbb{V}.
\end{equation*}
In the case that $\Gamma$ is connected, the image of $p_{\Gamma}: H^0(I(\Gamma), \underline{\mathbb{V}}) \rightarrow H^0(I(\Gamma), \mathcal{M})$ gives the trivial infinitesimal motions.

\begin{definition}
    Let $\Gamma$ be a hypergraph, and let $\mathcal{M}$ be a motion sheaf. We say that $\mathcal{M}$ is 
    \begin{itemize}
        \item \textit{independent} if $H^{1}(I(\Gamma), \mathcal{M}) = 0$,
        \item \textit{rigid} if $\Gamma$ is connected, and $p_{\Gamma}(H^{0}(I(\Gamma), \underline{\mathbb{V}})) = H^{0}(I(\Gamma), \mathcal{M})$,
        \item \textit{minimally rigid} if $\mathcal{M}$ is independent and rigid.
    \end{itemize}
\end{definition}
For motion sheaves of graph-of-groups realisations coming from bar-joint frameworks, these will notions correspond to independence, infinitesimal rigidity, and minimal infinitesimal rigidity respectively. 

\begin{theorem}\label{thm: Necessary condition}
Let $\Gamma=(V,E)$ be a hypergraph, and let $\mathcal{M}$ be a motion sheaf on $I(\Gamma)$ such that $\bigcap_{v\in V} \mathcal{S}(v) = 0$. If $\mathcal{M}$ is independent, then for any sub-hypergraph $\Gamma' \subseteq \Gamma$ with $\bigcap_{x\in V(\Gamma')} \mathcal{S}(x) = 0$, one has 
\begin{equation}\label{ineq}
\sum_{e \in E(\Gamma')}\left( \sum_{v\in e} \dim\left(\mathcal{M}(v)\right) - \dim\left(\mathcal{M}(e)\right) \right) \leq \sum_{v \in V(\Gamma')} \dim(\mathcal{M}(v)) - \dim(\mathbb{V}).
\end{equation}
If $\mathcal{M}$ is minimally rigid, then one additionally has:
\begin{equation} \label{eq: equality in nec_cond}
\sum_{e \in E}\left( \sum_{v\in e} \dim\left(\mathcal{M}(v)\right) - \dim\left(\mathcal{M}(e)\right) \right) = \sum_{v \in V} \dim(\mathcal{M}(v)) - \dim(\mathbb{V}).
\end{equation}
\end{theorem} 

\begin{proof}
We are given $H^{1}(I(\Gamma), \mathcal{M}) = 0$. By \Cref{lem: Stresses extend}, for any sub-hypergraph $\Gamma' = (V', E')$, one has $H^{1}(I(\Gamma'), \mathcal{M}) =0$. Thus, for any sub-hypergraph $\Gamma' \subseteq \Gamma$, applying \Cref{thm: Maxwell-rule} gives
\begin{align*}
\dim\left(H^{0}(I(\Gamma'), \mathcal{M}_{\vert I(\Gamma')})\right) &= \sum_{v \in V'} \dim(\mathcal{M}(v)) + \sum_{e \in E'} \dim(\mathcal{M}(e))  -\sum_{e\sim v \in I(\Gamma')} \dim(\mathcal{M}(v))\\
&= \sum_{v \in V'} \dim(\mathcal{M}(v)) + \sum_{e \in E'} \left( \dim(\mathcal{M}(e))  -\sum_{ v \in e} \dim(\mathcal{M}(v))\right).
\end{align*}
For $\Gamma'= \Gamma$, this gives the equality \eqref{eq: equality in nec_cond} when $\mathcal{M}$ is rigid.
If we can show that $\dim\left(H^{0}(I(\Gamma'), \mathcal{M}_{\vert I(\Gamma')})\right) \geq \dim(\mathbb{V})$ for sub-hypergraphs $\Gamma' =(V', E')$ with $\bigcap_{v\in V'} \mathcal{S}(v) =\{0\}$, we obtain the inequality \eqref{ineq}. It is clear that $\mathcal{M}_{\vert I(\Gamma')}$ is a motion sheaf on $I(\Gamma')$, and hence one has the map $p_{\Gamma'}:H^{0}(I(\Gamma'), \underline{\mathbb{V}}_{\vert I(\Gamma')}) \rightarrow H^{0}(I(\Gamma'), \mathcal{M}_{\vert I(\Gamma')})$. By the long exact sequence \eqref{les}, we have $$p_{\Gamma'}(H^{0}(I(\Gamma'), \underline{\mathbb{V}}_{\vert \Gamma'})) \cong \bigoplus_{C\in K(\Gamma')}  \frac{ \mathbb{V}}{ \bigcap_{v \in V(C)} \mathcal{S}(v)},$$

and we have
\begin{align*}
\dim\left(p_{\Gamma'}(H^{0}(I(\Gamma'), \underline{\mathbb{V}}_{\vert \Gamma'}))\right) &= \sum_{C\in K(\Gamma')} \dim(\mathbb{V}) - \dim(\bigcap_{v\in C} \mathcal{S}(v))  \\
&= \sum_{C\in K(\Gamma')} \codim(\bigcap_{v\in C} \mathcal{S}(v))\\
& \geq \codim(\bigcap_{v\in V'} \mathcal{S}(v)) = \dim(\mathbb{V}).
\end{align*}
This completes the proof of the theorem.
\end{proof}

If $s < \frac{(r-1)n}{r}$, for sheaves in $M_{s, n}(\Gamma)$, the necessary condition is given by $(n-s, n)-$sparsity of the hypergraph with edges copied a given number of times. For higher values of $s$, one can derive similar conditions, though the condition will not be matroidal.  
\begin{corollary}\label{cor: necessary_concrete}
    Let $\Gamma=(V, E)$ be an $r$-regular hypergraph. If $s < \frac{(r-1)n}{r}$, and if $\mathcal{M}\in M_{s,n}(\Gamma)$ is such that $\bigcap_{v\in X} \mathcal{S}(v) = 0$ for any subset $X\subseteq V$ with $|X|\geq r$, and if $\mathcal{M}$ is independent, then
    \begin{align*}
        \left((r-1)n -rs\right) |E'| &\leq (n-s)|V'| - n \textup{ for all sub-hypergraphs $(V', E') \subseteq \Gamma$}.
    \end{align*}
    If $\mathcal{M}$ is minimally rigid, then additionally, one has
    \begin{align*}
        \left((r-1)n -rs\right) |E| &= (n-s)|V| - n,
    \end{align*}
\end{corollary}
\begin{proof}
    We have
    \begin{align*}
       \dim(\mathcal{M}(v))& = n -\dim(\mathcal{S}(v)) = n- s\\
        \dim(\mathcal{M}(e))& = n -\dim(\bigcap_{v\in e} \mathcal{S}(v)) = n
    \end{align*}
    and thus 
    \begin{equation*}
       \sum_{ v \in e} \dim(\mathcal{M}(v)) - \dim(\mathcal{M}(e))= r(n -s) -n = (r-1) n - rs.
    \end{equation*}
    The result then follows from \Cref{thm: Necessary condition}.
\end{proof}
For graphs, one has $r=2$, and then the condition gives $(n-s, n)$-sparsity of $(n-2s)\Gamma$. 

\subsection{Genericity for motion sheaves}

In this section, we shall show that given a hypergraph $\Gamma$, the functions 
\begin{align*}
    h^{0}: M_{s,n}(\Gamma) \rightarrow \mathbb{N}: \mathcal{M} \rightarrow \dim(H^{0}(I(\Gamma), \mathcal{M}))\\
    h^{1}: M_{s,n}(\Gamma) \rightarrow \mathbb{N}: \mathcal{M} \rightarrow \dim(H^{1}(I(\Gamma), \mathcal{M}))
\end{align*}
are upper semi-continuous functions with respect to the Zariski topology. From this, it follows that the cohomology groups have the same dimension for a nonempty Zariski open subset of $M_{s,n}$.

\begin{definition}
   Let $X$ be a topological space, and let $Z\subseteq \mathbb{R}$. A function $f: X\rightarrow Z$ is said to be lower semi-continuous if $f^{-1}((a, \infty))$ is open for each $a\in Z$. A function $f: X\rightarrow Z$ is said to be upper semi-continuous if $f^{-1}((-\infty, a))$ is open for each $a\in Z$. 
\end{definition}

\begin{proposition}\label{prop: Rank-lower semicont}
    Let $X$ be a real algebraic variety, and let $f_{1,1} ,\dots, f_{n,m}\in \mathbb{R}(X)$ be  $n\times m$ rational functions, all defined on an open subset $U\subset X$. Then 
    \begin{equation*}
        \textup{rk}: U\rightarrow \mathbb{N}:x \mapsto 
         \textup{Rank}\left(
        \begin{bmatrix}
            f_{1,1}(x) &\dots &f_{1,m}(x)\\
            \vdots &\ddots &\vdots\\
            f_{n,1}(x) &\dots &f_{n,m}(x)
        \end{bmatrix}\right)
    \end{equation*}
    is a lower semi-continuous function with respect to the Zariski topology on $U$.
\end{proposition}

Commonly, one uses Plücker coordinates on Grassmannians. We will instead use a description of the Grassmannian as a space of projection operators, which gives it the structure of an affine real algebraic variety, see \cite[Section 3.4.2]{Bochnak1998}. For a comparison to Plücker coordinates, see \cite{devriendt2025two}, where, among other things, it is proven that the description that follows is birationally equivalent to the usual projective variety. For any subspace $S\subseteq \mathbb{R}^{n}$ of dimension $s$, let $P: \mathbb{R}^{n}\rightarrow \mathbb{R}^{n}$ be the orthogonal projection with image $S$. One then has the conditions
$ P=P^{t}, P^2=P$, and $\textup{Tr}(P) = s$.
These three equations define the Grassmannian $\textup{Gr}(s, \mathbb{R}^{n})$ as an affine subvariety of $\mathbb{R}^{n\times n}$.

\begin{lemma}\label{lem: Intersections_of_subspaces}
    For any subset $I\subseteq \{1, \dots, m\}$, the function 
    \begin{equation*}
    \textup{Gr}(s, \mathbb{V})^{m} \rightarrow \mathbb{N}: (S_1, \dots, S_m) \mapsto \dim(\cap_{ j \in I}(S_j))
    \end{equation*}
    is upper semi-continuous for the Zariski topology on $\textup{Gr}(s, \mathbb{V})^{m} $.
\end{lemma}
\begin{proof}
Let $I = \{ k_1, \dots, k_\ell\} \subseteq \{1, \dots, m\}$. Fixing an inner product and a basis on $\mathbb{V}$, we view $\mathbb{V}$ as being $\mathbb{R}^{n}$, and we can use the description of the Grassmannian as a space of projection operators given above. We will consider the kernel of the following map
\begin{equation*}
    \varphi: S_{k_1} \times \dots \times S_{k_\ell} \mapsto (\mathbb{R}^{n})^{\binom{\ell}{2}} : \left(w_{k_1}, \dots, w_{k_\ell}\right) \mapsto (w_{k_{i}} - w_{k_j})_{k_i < k_j},
\end{equation*}
as the choices of $S_{k_i}$ vary. One can immediately see that $\ker(\varphi)\cong \bigcap_{k_i \in I} S_{k_i}$.

To understand $\ker(\varphi),$ we let $\tilde{\varphi}: \mathbb{R}^{n} \times \cdots \times  (\mathbb{R}^{n})^{\binom{l}{2}} \rightarrow \mathbb{R}^{\binom{n}{2}}: (w_{k_1}, \dots, w_{k_\ell}) \mapsto \left(P_{k_i}(w_{k_i}) -  P_{k_j}(w_{k_j})\right)_{k_i < k_j},$ where $P_{k_i}$ is the orthogonal projection onto $S_{k_i}$. It is clear that
\begin{equation*}
    \ker(\tilde{\varphi}) \cong \ker(\varphi) \oplus S_{k_1}^{\perp} \oplus  \dots \oplus  S_{k_\ell}^{\perp}.
\end{equation*}
    In a matrix, the map $\tilde{\varphi}$ is given by   \begin{equation*}
    \begin{bmatrix}
        P_{k_1} & -P_{k_2} & 0 & 0 & \dots & 0 &0 \\
        P_{k_1} & 0 & -P_{k_3} & 0 & \dots & 0 &0 \\
\vdots & \vdots&\vdots &\vdots & \ddots & \vdots&\vdots \\
        0 &  0 & 0 & 0 & \dots & P_{k_{\ell-1}} & -P_{k_{\ell}}\\
    \end{bmatrix}.
    \end{equation*}
    Then, by \Cref{prop: Rank-lower semicont}, the rank of this matrix is lower semi-continuous as a function of $P_{k_i}$, which are precisely the coordinates of the Grassmannian. Thus $\dim(\ker(\tilde{\varphi})) = n\ell - \textup{rk}(\tilde{\varphi})$ is upper semi-continuous. Then, since 
    \begin{equation*}
        \dim(\ker(\tilde{\varphi})) = \dim(\ker(\varphi)) + \sum_{j=1}^{\ell}  \dim(S_{k_j}^{\perp}),
    \end{equation*}
    and since $\dim(S_{k_j}^{\perp})$ is constant, one sees that $\dim(\ker(\varphi))$ is also upper semi-continuous. 
\end{proof}

\begin{example}
    Let us consider the case where $s = 1$. In that case, one has $\dim(S_1 \cap S_2) = \{0\}$ if and only if $S_1\neq S_2$. With the notation of the lemma, this implies that the dimension of the intersection can only increase on the closed subsets defined by $S_i = S_j$ for all $i,j \in I$. 
\end{example}

\begin{theorem}\label{thm: Genericity_generalised_frameworks.}
Let $\Gamma$ be a hypergraph. Let $U \subseteq Gr(s,\mathbb{V})^{|V|}$ be a semi-algebraic subset such that for any $x=(S_1, \dots, S_{|V|}) \in U$, and for any hyper-edge $u_1, \dots, u_{k} \in E$, one has that $\dim(\bigcap_{i=1}^{k}S_{u_i})$ is constant as a function of $S_{u_i} \in U$. Identifying the set $M_{s,n}(\Gamma)$ with $\textup{Gr}(s,\mathbb{V})^{|V|}$, we treat points in $x\in U$ as sheaves. We define $h^{0}: U \rightarrow \mathbb{N}$ and $h^{1}: U \rightarrow \mathbb{N}$ by
\begin{align*}
    h^{0}(\mathcal{M})& := \dim(H^{0}(I(\Gamma), \mathcal{M})) & h^{1}(\mathcal{M}) :=  \dim(H^{1}(I(\Gamma), \mathcal{M})).
\end{align*}
 The functions $h^{0}(\mathcal{M})$ and $h^{1}(\mathcal{M})$ are upper semi-continuous with respect to the Zariski topology restricted to $U$. 
 In particular, there is a Zariski open subset of $\mathcal{M}\in U$ such that $h^{0}(\mathcal{M})$ and $h^{1}(\mathcal{M})$ attain the minimum dimension.
\end{theorem}
\begin{proof}
 We use the isomorphism $\textup{Gr}(s, \mathbb{R}^n) \rightarrow  \textup{Gr}(n-s, \mathbb{R}^n)$. Using the realisation as projection operators, it is given by
\begin{equation*}
   \varphi: \textup{Gr}(s, \mathbb{R}^n) \rightarrow  \textup{Gr}(n-s, \mathbb{R}^n): P \mapsto I - P.
\end{equation*}

Under this isomorphism, one recovers the space $S \in \textup{Gr}(s, \mathbb{R}^{n})$ which $P$ represents as the kernel of the orthogonal projection $P$, and one can think of the projection $P$ as being a projection onto $\mathbb{R}^{n}/S$. 
Thus, let $(S_{v})_{v\in V} \in U$, with associated orthogonal projections $P_{v}$ onto $S_{v}^{\perp}$. Consider the sheaf $\mathcal{F}$ on $I(\Gamma)=(V\cup E, I)$ with $\mathcal{F}(e) = \mathcal{F}(v) = \mathbb{R}^{n}$. We put
    \begin{equation*}
        r^{e}_{v_i \sim e} = P_{v_i}
    \end{equation*}
    When writing the operator $d: \bigoplus_{x\in V\cup E} \mathcal{F}(x) \rightarrow \bigoplus_{i \in I} \mathcal{F}(i)$ as a matrix, the entries are given by the entries of the $P_{v}$, which are coordinates of the Grassmannian. By \Cref{prop: Rank-lower semicont}, the dimension of $H^{0}(I(\Gamma), \mathcal{F})$ is upper semi-continuous on $U$.
    One has
    \begin{equation*}
        H^{0}(I(\Gamma),\mathcal{F}) \cong \left( \oplus_{v\in V} S_{v} \right) \oplus \left( \oplus_{e\in E} (\cap_{v\in e} S_{v}) \right) \oplus H^{0}(I(\Gamma), \mathcal{M}).
    \end{equation*}
    Then, since for $(S_{v})_{v\in V}\in U$ the ranks of $\cap_{v \in e} S_{v}$ are the same minimal number, we see that $h^{0}(\mathcal{M})$ is upper semi-continuous. By \Cref{thm: Maxwell-rule}, $h^{1}(\mathcal{M})$ is as well. 
\end{proof}

\begin{remark}
     \Cref{thm: Genericity_generalised_frameworks.} is not true if we replace $U$ by $\textup{Gr}(s, \mathbb{V})^{|V|}$. For instance, let $K_3$ be the complete graph on $3$ vertices $\{v_1,v_2,v_3\}$. Let $\mathbb{V}=\mathbb{R}^{3}$ and $s=1$. Pick $\mathcal{S}(v_1) =\mathcal{S}(v_2)  =\mathcal{S}(v_3) = \langle (1,0,0) \rangle$. One can compute that $h^{0}(\underline{\mathbb{R}}^{3}/ \mathcal{S}) = 2$. One can also compute that for any $\mathcal{S}\in S_{s,n}(K_3)$ with $\sum_{i=1}^{3} \mathcal{S}(v_i) = \mathbb{R}^{3}$, one has $h^{0}(\underline{\mathbb{R}^{3}}/ \mathcal{S}) = 3$. One can in fact also have $h^{0}(\underline{\mathbb{R}^{3}} / \mathcal{S}) = 4$ in the case that $\mathcal{S}(v_1), \mathcal{S}(v_2)$ and $\mathcal{S}(v_3)$ are all distinct, but when $\sum_{i=1}^{3} \mathcal{S}(v_i)$ has dimension $2$. On the set $U=\{ \mathcal{S} \in \textup{Gr}(1, \mathbb{R}^{3})^{3}~\vert ~\mathcal{S}(v_i)\neq \mathcal{S}(v_j) \}$, \Cref{thm: Genericity_generalised_frameworks.} applies, and we see that the dimension does increase on a certain closed subset $Y\subseteq 
    U$, but for a dense open subset of $\textup{Gr}(1, \mathbb{R}^{3})^{3}$ one has that $h^{0}(\underline{\mathbb{R}^{3}}/\mathcal{S}) =3$.
\end{remark}

\subsection{Associated sheaves on 
multigraphs}\label{assoc_sheaves}
In this section, given $\mathcal{M} \in M_{s,n}(\Gamma)$ for a graph $\Gamma$, we define sheaves on $(n-2s)\cdot \Gamma$ that have the same cohomology groups.

\begin{proposition}\label{prop: F from M}
    Let $\Gamma=(V, E)$ be a graph. Let $\mathcal{S} \in S_{s,n}(\Gamma)$ and let $\mathcal{M} = \underline{\mathbb{V}}/\mathcal{S}$. Define a sheaf $\mathcal{F}$ on $\Gamma$ by
    \begin{align*}
        \mathcal{F}(v) &= \mathcal{M}(v) = \mathbb{V}/\mathcal{S}(v) \textup{ for all }v\in V\\
        \mathcal{F}(e) &= \mathbb{V}/\left( \mathcal{S}(v) +\mathcal{S}(u)\right)\textup{ if } e = vu,
    \end{align*}
    and with maps $r^{v}_{e}:\mathcal{F}(v)\rightarrow \mathcal{F}(e)$ being the restriction modulo $\mathcal{S}(v) +\mathcal{S}(u)$.
    We have
        \begin{align*}
        H^{0}(\Gamma, \mathcal{F})&= H^{0}(I(\Gamma), \mathcal{M}),\\
        H^{1}(\Gamma, \mathcal{F})&= H^{1}(I(\Gamma),  \mathcal{M}).
    \end{align*}
\end{proposition}
\begin{proof} Consider first a sheaf $\mathcal{G}$, defined on $I(\Gamma)$, by
    \begin{align*}
        \mathcal{G}(v) &= \mathcal{M}(v) \textup{ for all }v\in V\\
        \mathcal{G}(e) &= \mathbb{V}/\left( \mathcal{S}(v) +\mathcal{S}(u)\right)\textup{ for all } vu\in E \\
        \mathcal{G}(v\sim e) &= \mathbb{V}/\left( \mathcal{S}(v) +\mathcal{S}(u)\right)\textup{ for all } vu\in E.
    \end{align*}
    Consider the map $\alpha: \mathcal{M} \rightarrow \mathcal{G}$ given by
    \begin{align*}
        \alpha_v &= \textup{id} \textup{ for all }v\in V,\\
        \alpha_e &= \pi_e \textup{ for all } e\in E,\\
        \alpha_{v \sim e} &= \pi^{v}_e \textup{ for all } v\sim e \in I,
    \end{align*}
where $\pi_e: \mathbb{V}/\mathcal{S}(e)  \rightarrow \mathbb{V}/\left( \mathcal{S}(v) +\mathcal{S}(u)\right)$ and $\pi^{v}_e: \mathbb{V}/\mathcal{S}(v) \rightarrow \mathbb{V}/\left( \mathcal{S}(v) +\mathcal{S}(u)\right)$ are the reductions modulo $\mathcal{S}(v) +\mathcal{S}(u)$. Then we have a short exact sequence of sheaves:
\begin{align}\label{eq: les_prop_F-M}
    0\rightarrow \ker(\alpha) \rightarrow \mathcal{M} \rightarrow\mathcal{G} \rightarrow 0.
\end{align}
One has $\ker(\alpha)(v) = \{0\}, \ker(\alpha)(e) = \frac{\mathcal{S}(v) + \mathcal{S}(u) }{\mathcal{S}(v) \cap \mathcal{S}(u)}$, and $ \ker(\alpha)(v \sim e) = \frac{\mathcal{S}(v) + \mathcal{S}(u)}{\mathcal{S}(v)}$.  It then follows that
        \begin{align*}
        H^{0}(I(\Gamma), \ker(\alpha))&= 0,\\
        H^{1}(I(\Gamma), \ker(\alpha))&= 0.
    \end{align*}
Then, the long exact sequence associated with the short exact sequence \eqref{eq: les_prop_F-M} implies that
        \begin{align*}
        H^{0}(I(\Gamma), \mathcal{G})&\cong H^{0}(I(\Gamma), \mathcal{M}),\\
        H^{1}(I(\Gamma), \mathcal{G})&\cong H^{1}(I(\Gamma),  \mathcal{M}).
    \end{align*}
To finish the proof, it suffices to show 
\begin{align*}
     H^{0}(I(\Gamma), \mathcal{G})&\cong H^{0}(\Gamma, \mathcal{F}),\\
    H^{1}(I(\Gamma), \mathcal{G})&\cong H^{1}(\Gamma, \mathcal{F}),
\end{align*}
which follows in a straightforward way by simply using the definitions. 
\end{proof}

We now discuss how to define a sheaf on a multigraph related to any sheaf on a graph with the same cohomology groups. Let $\Gamma=(V, E)$ be a graph, let $\mathcal{F}$ be a sheaf on $\Gamma$, and fix a basis $\mathcal{B}_e = (b_1, \dots, b_{m_e})$ for each $\mathcal{F}(e)$ for $e\in E$. One can thus write the maps defining $\mathcal{F}$ as $r^{v}_{e}(x) =f_{e, 1}(x)b_{1} + \dots + f_{e, m_e}(x)b_{m_e}$ for some $f_{e, i} \in \mathbb{V}^{*}$.

Define $\tilde{\Gamma}$ to be the multi-graph with vertex set $V$, and edge set having $\dim(\mathcal{F}(e))$ copies of each edge. 

Label these edges $(e, 1), \dots, (e, m_e)$. Define a sheaf $\tilde{\mathcal{F}}$ on $\tilde{\Gamma}$ by
\begin{align*}
    \tilde{\mathcal{F}}(v) &= \mathcal{F}(v),\\
    \tilde{\mathcal{F}}(e, i) &= \mathbb{R}  \text{ for each edge } (e, i) \in E.
\end{align*}
the maps $r^{v}_{e, i}$ are given by $f_{e, i}$ defined above. One computes easily that: 
\begin{align*}
    H^{0}(\tilde{\Gamma}, \tilde{\mathcal{F}}) &\cong  H^{0}(\Gamma, \mathcal{F})\\
    H^{1}(\tilde{\Gamma},\tilde{\mathcal{F}}) &\cong  H^{1}(\Gamma, \mathcal{F}).
\end{align*}

We apply this construction to the sheaf $\mathcal{F}$ from \Cref{prop: F from M}. For any subspace $S\leq \mathbb{V},$ of dimension $2s$, and any $n-2s$ linearly independent linear forms $\alpha_1, \dots, \alpha_{n-2s} \in \mathbb{V}^{*}$ such that $\alpha_{i}(S) = 0$ for $i\in \{1, \dots, n-2s\}$, the map
\begin{equation*}
\mathbb{V}\rightarrow \mathbb{R}^{n-s}: x \mapsto (\alpha_{1}(x), \dots \alpha_{n-2s}(x))
\end{equation*}
gives a coordinate expression for the reduction $\mathbb{V}\rightarrow \mathbb{V}/S$. Thus, letting $\lambda \cdot \Gamma$ be the multigraph with $\lambda_e$ copies of each edge, where $\lambda_e = \dim(\mathbb{V}/(\mathcal{S}(v) + \mathcal{S}(u)))$, we pick any basis $\alpha_{e,1}, \dots, \alpha_{e, \lambda_e}$ of the annihilator of $\mathcal{S}(v) + \mathcal{S}(u)$ on each edge $vu$, and then we define $\tilde{\mathcal{F}}(v)$ to be equal to $\mathbb{V}/\mathcal{S}(v)$ and $\tilde{\mathcal{F}}(e)=\mathbb{R}$. The linear map $r^{v}_{(e,i)}$ is given by $[w_{v}] \mapsto \alpha_{e, i}(w_v)$ for each edge, which is well defined since $\mathcal{S}(v) \subseteq \mathcal{S}(v) + \mathcal{S}(u)\subseteq \ker(\alpha_{e})$. With this notation, we have:

\begin{proposition}\label{prop: construction-F-tilde}
    Let $s\leq n/2$, let $\mathcal{S} \in S_{s,n}(\Gamma)$ and let $\mathcal{M} = \underline{\mathbb{V}}/\mathcal{S}$ with  $\mathcal{S}(v) \cap \mathcal{S}(u) = 0$ for all edges $uv\in E$. Suppose that $\tilde{\mathcal{F}}$ as above using any basis $\alpha_{e, 1}, \dots \alpha_{e, n-2s}$ for each edge $e\in \Gamma$. Then
        \begin{align*}
        H^{0}((n-2s)\cdot \Gamma, \tilde{\mathcal{F}})&= H^{0}(I(\Gamma), \mathcal{M})\\
        H^{1}((n-2s)\cdot\Gamma, \tilde{\mathcal{F}})&= H^{1}(I(\Gamma),  \mathcal{M}).
    \end{align*}
\end{proposition}

We want to equip the set of possible sheaves of the form $\tilde{\mathcal{F}}$ with the structure of a real algebraic variety, as we have done in \Cref{sec: genericity for graph of groups realisations}. Given a multigraph $\Gamma=(V,E)$, we will thus consider the sheaves such that
\begin{align*}
    \mathcal{F}(v) &= \mathbb{V}/S_v &\textup{ for $v\in V$,}\\
    \mathcal{F}(e) &= \mathbb{V}/\textup{ker}(\alpha_{e}) \cong \mathbb{R} &\textup{ for $e\in E$},
\end{align*}
where $\dim(S_v) = s, \dim(\mathbb{V}) = n$ and $\alpha_e \in \mathbb{V}^{*}$ with $\alpha_{e}(S_v)=0$ if $v\in e$. 
One can see the set of sheaves arising in this way as a real algebraic variety. Namely, it consists of 
\begin{equation*}
 \left( (S_{v})_{v\in V}, (\alpha_{e})_{e\in E}\right) \in \textup{Gr}(s, \mathbb{V})^{|V|} \times (\mathbb{V}^{*})^{E}
\end{equation*}
such that $\alpha_{e}(S_v) = 0$ if $v\in e$. For a multigraph, define $Z_{s, n}(\Gamma)$ to be the set of sheaves such that this holds, and additionally, the subspaces at the vertices have an empty intersection, i.e:
\begin{equation*}
    Z_{s, n}(\Gamma):= \{ \left((S_{v})_{v\in V}, (\alpha_{e})_{e\in E}\right)  \in \textup{Gr}(s, \mathbb{V})^{|V|} \times (\mathbb{V}^{*})^{E}~\vert ~ \alpha_{e}(S_{v} ) = 0 \textup{ if $v\in e$ and } \forall u, v \in V: S_{v} \cap S_{u} = \{0\}    \}.
\end{equation*}
The condition that $\alpha(S) = 0$ is an algebraic condition, which can be expressed as $\alpha \circ P = 0$ with the coordinates expressing $\textup{Gr}(s, n)$ as projection operators. We remark that the definition only makes sense when $s  \leq n/2$, since otherwise the spaces $S$ necessarily have nonempty intersections. Similarly to $\mathcal{M}$, for every $\mathcal{F} \in Z_{s,n}(\Gamma)$, we have a short exact sequence of sheaves
\begin{equation}\label{eq: les-ass}
    0 \rightarrow \mathcal{S} \rightarrow \underline{\mathbb{V}} \rightarrow \mathcal{F} \rightarrow 0. 
\end{equation}
Here $\mathcal{S}(v)$ is an $s$-dimensional subspace of $\mathbb{V}$ for each $v\in V$, and $\mathcal{S}(e) = \ker(\alpha_{e})$ for all $e\in E$. The maps $ \mathcal{S}(v) \rightarrow \mathcal{S}(e)$ are given by the inclusions. We will denote these sheaves by $\mathcal{F} = \underline{\mathbb{V}}/\mathcal{S}$ in analogy to what we have done for the motion sheaves. We will now study some of the geometric properties of $Z_{s,n}$, but we first need a simple fact. 

\begin{lemma}\label{lem: sum is rational map}
    Let $s < n/2$. The map 
    \begin{equation*}
        Sum: \textup{Gr}(s, n)\times \textup{Gr}(s, n) \rightarrow \textup{Gr}(2s, n): (V_1, V_2)\mapsto V_1 + V_2
    \end{equation*}
    is a rational map.
\end{lemma}
\begin{proof}
    The map is defined on the subset $O\subseteq\textup{Gr}(s, n)\times \textup{Gr}(s, n)$ of pairs $(V_1,V_2)$ such that $V_1\cap V_2 = 0$, which is Zariski open by \Cref{lem: Intersections_of_subspaces}. Now, using Plücker coordinates, one has that the map is given by
    \begin{equation*}
        ((v_1\wedge \dots \wedge v_k), (w_1\wedge \dots \wedge w_k))\mapsto (v_1\wedge \dots \wedge v_k \wedge w_1\wedge \dots \wedge w_k)
    \end{equation*}
    and it is easy to see that, in terms of the Plücker coordinates, this map is given by bihomogeneous polynomials, giving a regular map on $O$. 
\end{proof}

\begin{lemma}\label{lem: Properties map Z -> M}
     The map $f: Z_{s,n}((n-2s)\cdot \Gamma) \rightarrow M_{s,n}(\Gamma)$ given by $((P_{v})_{v\in V}, (\alpha_{e})_{e\in E})\mapsto (P_{v})_{v\in V}$ gives $Z_{s,n}((n-2s)\cdot \Gamma )$ the structure of a bundle with fibre being a product of real vector spaces over a Zariski open subset of $M_{s,n}(\Gamma)$, and one has
    \begin{align*}
        \min_{\mathcal{F} \in f^{-1}(\mathcal{M})}\dim\left( H^{0}((n-2s)\cdot \Gamma, \mathcal{F}\right) &= \dim\left( H^{0}(I(\Gamma), \mathcal{M})\right),\\
        \min_{\mathcal{F} \in f^{-1}(\mathcal{M})}\dim \left( H^{1}((n-2s)\cdot \Gamma, \mathcal{F}\right)&= \dim\left( H^{1}(I(\Gamma), \mathcal{M})\right).
    \end{align*}
    Moreover, the variety $Z_{s,n}$ is irreducible and nonsingular.
\end{lemma}
\begin{proof}
    For $P_v$, denote the represented subspace by $S_v$.  Given $(P_{v})_{v\in V}$, we see that $f^{-1}((P_{v})_{v\in V})$ is given by those $((\alpha_{e})_{e\in E}, (P_{v})_{v\in V})$, such that $\alpha_{e} \in \textup{Ann}(S_{v} + S_{u})$ for $e=uv$. By \cite[proposition 12.1.4]{Bochnak1998}, the variety $E_{2s,n}^{\perp} = \{ (\alpha,P)\in \mathbb{V}^{*}\times \textup{Gr}(2s, n)~\vert~ \alpha \circ P=0\}$ is an algebraic vector bundle over $ \textup{Gr}(2s, n)$. 
    For each edge, pull back $E_{2s, n}^{\perp}$ under the map $M_{s, n}(\Gamma) \rightarrow \textup{Gr}(2s, n): (S_{x})_{x\in V} \mapsto (S_{u} + S_{v})$, and then take the direct sum of all the resulting bundles. One then obtains $Z_{s, n}(\Gamma)$, which shows that $Z_{s,n}(\Gamma)$ is a vector bundle. The variety $Z_{s,n}(\Gamma)$ is irreducible since it is the total space of a vector bundle over an irreducible base. This is easily seen to be an irreducible variety by covering $Z_{s, n}(\Gamma)$ with the local trivialisations $U_i\times \mathbb{R}^{K}$ for Zariski open sets $U_i \subseteq M_{s,n}(\Gamma)$ (see \cite[Definition 12.1.1]{Bochnak1998}), which are seen to be irreducible using \cite[Theorem 2.8.3 (iii)]{Bochnak1998}. Since $\textup{Gr}(s, n)$ is nonsingular, so is $Z_{s,n}(\Gamma)$.\\
    We saw in \Cref{prop: construction-F-tilde} that if, for each $e\in \Gamma$, the $\alpha_{e,1},\dots, \alpha_{e,n-2s}$ are a basis for the annihilator of $S_v +S_W$, then $H^{0}((n-2s)\cdot \Gamma, \mathcal{F}) \cong H^{0}( I(\Gamma), \mathcal{M})$. Otherwise, one may still define an injective map $\varphi: H^{0}(I(\Gamma), \mathcal{M}) \rightarrow H^{0}((n-2s)\cdot \Gamma, \mathcal{F})$, which is explicitly given by $(\left(W_{v}\right)_{v \in V},\left(W_{e}\right)_{e \in E} ) \mapsto (W_{v})_{v\in V}$, and then one has $\dim\left(H^{0}((n-2s)\cdot \Gamma, \mathcal{F})\right)\geq \dim\left(H^{0}(I(\Gamma), \mathcal{M}))\right)$. This shows the claim about the minimum for $h^{0}$. Using $h^{0}(\mathcal{F}) - h^{1}(\mathcal{F}) =  \mathcal{F}(V) - \mathcal{F}(E),$  
    the same follows for $h^{1}$.
\end{proof}

Similarly to \Cref{thm: Genericity_generalised_frameworks.}, we have the following.
\begin{theorem}\label{thm: Genericity_F}
Let $\Gamma=(V,E)$ be a multigraph. Define $h^{0}: Z_{s, n}(\Gamma) \rightarrow \mathbb{N}$ and $h^{1}: Z_{s, n}(\Gamma) \rightarrow \mathbb{N}$ by
\begin{align*}
    h^{0}(\mathcal{F})& := \dim(H^{0}(\Gamma, \mathcal{F})) & h^{1}(\mathcal{F}) :=  \dim(H^{1}(\Gamma, \mathcal{F})).
\end{align*}
 The functions $h^{0}$ and $h^{1}$ are upper semi-continuous with respect to the Zariski topology.
\end{theorem}
\begin{proof}
    The proof is analogous to that of \Cref{thm: Genericity_generalised_frameworks.}. For each $(P_{v})_{v\in V}, (\alpha_{e})_{e\in E}$, where $P_v$ represents $S_{v}\in \textup{Gr}(s, \mathbb{V})$, define the $|E|\times n\cdot |V|$ matrix $R$ with rows $\alpha_{e} \circ (\textup{Id} - P_{v})$.
    One can verify that
    \begin{equation*}
        \ker(R) \cong H^{0}(\Gamma, \mathcal{F}) \oplus \bigoplus_{v\in V} S_{v}
    \end{equation*}
    Since the dimensions of $S_{v}$ are constant, the theorem follows by \Cref{prop: Rank-lower semicont}, using the fact that $h^{0}(\mathcal{F}) - h^{1}(\mathcal{F}) =\dim(\mathcal{F}(V)) - \dim(\mathcal{F}(E))$ 
\end{proof}

\subsection{Inductive constructions for associated sheaves}

\begin{definition}
    Let $\Gamma=(V, E)$ be a multi-graph. A $d$-dimensional $k$-extension is a new graph $\Gamma'= (V\cup \{v_*\},  E')$ resulting from removing $k$ edges from $\Gamma$, adding a new vertex $v_{*}$, adding edges to the endpoints of the deleted edges, and adding $d - k$ additional edges. More formally, one can write $E' = \left(E\setminus \{e_1, \dots, e_k\}\right) \cup \{f_1, \cdots, f_{d+k}\}$. If the edges $e_j$ are given by $\{u_j, u_j'\}$, then the new edges are given by 
    \begin{align*}
        f_{2j-1} &= v_* u_j & \textup{for $j \leq k$},\\
        f_{2j} &= v_* u_j' & \textup{for $j \leq k$}, \\
        f_{2k+j} &= v_* v_j \textup{ for some $v_j \in V$} & \textup{for $j \leq d-k$}.
    \end{align*}
\end{definition}

We will provide an algebraic condition for the extension moves to preserve independence. First, let us make some straightforward observations on how to do computations in $H^{1}(\Gamma, \underline{\mathbb{V}})$. 

Suppose that $\Gamma=(V,E)$ is a graph, with some orientation $[v: e]$ on the edges, and that $(w_{e})_{e\in E}, (w_{e}')_{e\in E} \in \mathbb{V}^{|E|}$ define elements of $H^{1}(\Gamma, \underline{\mathbb{V}})$. By definition, they define the same element of $ H^{1}(\Gamma, \underline{\mathbb{V}})$ if and only if
\begin{equation*}
    (w_{e} - w_{e}')\in \im(d).
\end{equation*}
Taking some $W = \sum_{v\in V} w_v \in \bigoplus_{v\in V}\mathcal{F}(v)$ such that $w_{v}= 0$ for all $v\in V$ with $v\neq v_0$, and $w_{v_0} \neq 0$, one finds that for any cochain $\sum_{e\in E} w_{e}$, that in $H^{1}(\Gamma, \underline{\mathbb{V}})$
\begin{equation} \label{eq: rule for h1}
    \sum_{e\in E} w_e = \sum_{e\in E} w_e + \sum_{e: v_0\in e} [v_0: e] w_{v_0}.
\end{equation}
One finds such an equation for all $v\in V$, and we will repeatedly apply this rule. Another way to write this rule is that in the edge neighbourhood $N(v) =\{e\in E~ \vert ~ v\in  e\}$ of any vertex $v$, one has, denoting the outgoing edges by $N^{-}(v) =\{e\in E~ \vert ~ v = o(e)\}$, and the incoming edges by $N^{+}(v) =\{e\in E~ \vert ~ v = t(e)\}$, one has
\begin{equation*}
    \sum_{e\in E} w_e  = \sum_{e\in E\setminus N(v)} w_e + \sum_{e\in N^{+}(v)} (w_e - \tilde{w}) + \sum_{e\in N^{-}(v)} (w_e + \tilde{w}).
\end{equation*}
When the underlying graph $\Gamma$ is ambiguous, we will write $N_{\Gamma}(v)$. We emphasise that $N(v)$ is a subset of edges, in contrast with the usual use of this notation.  

One can see that if two cochains $w_{e}, w'_{e}$ define the same element of $H^{1}(\Gamma, \underline{\mathbb{V}})$, then one can show that they are equal by repeatedly applying equation \eqref{eq: rule for h1}. See also \Cref{fig: computation_rule} for an illustration. 

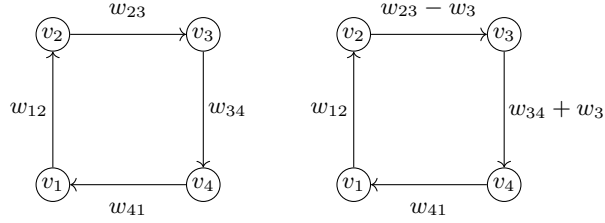
\begin{figure}[h]
\centering
\begin{tikzpicture}[
    every node/.style={circle, draw, fill=white, inner sep=1pt, minimum size=5pt, font=\small}
]
    \node (a) at (0,0) {$v_1$};
    \node (b) at (0,2) {$v_2$};
    \node (c) at (2,2) {$v_3$};
    \node (d) at (2,0) {$v_4$};

    \draw [->] (a) -- (b)  node[midway, left, draw=none, fill=none] {$w_{12}$};
    \draw [->] (b) -- (c)  node[midway, above, draw=none, fill=none] {$w_{23}$};
    \draw [->] (c) -- (d)  node[midway, right, draw=none, fill=none] {$w_{34}$};
    \draw [->] (d) -- (a)  node[midway, below, draw=none, fill=none] {$w_{41}$};

    \node (a') at (4,0) {$v_1$};
    \node (b') at (4,2) {$v_2$};
    \node (c') at (6,2) {$v_3$};
    \node (d') at (6,0) {$v_4$};

    \draw [->] (a') -- (b')  node[midway, left, draw=none, fill=none] {$w_{12}$};
    \draw [->] (b') -- (c');
    
    \node[ draw=none, fill=none] at (5, 2.35) {$w_{23} - w_3$};
    \draw [->] (c') -- (d')  node[midway, right, draw=none, fill=none] {$w_{34} + w_3$};
    \draw [->] (d') -- (a')  node[midway, below, draw=none, fill=none] {$w_{41}$};
\end{tikzpicture}
\caption{ Two cochains defining the same element of $H^{1}(\Gamma, \underline{\mathbb{V}})$. They are the same by equation \eqref{eq: rule for h1}}.
\label{fig: computation_rule}
\end{figure}

\begin{theorem}\label{Extension-for-sheaves}
    Suppose that $s < n/2$, and let $d = n-s$.
    Let $\Gamma =(V, E)$ be a multi-graph, and suppose $\mathcal{F} = \underline{\mathbb{V}}/\mathcal{S} \in Z_{s, n}(\Gamma)$ such that $h^{1}(\mathcal{F}) = 0$. Suppose that $\Gamma'$ results from applying a $d$-dimensional $k$-extension, with new vertex $v_{*}$, new edges $f_{1} =u_{1}v_{*}, f_{2}=u_{1}'v_{*},\dots, f_{2k +1}=v_{1}v_*, \dots, f_{d+k}=v_{d-k}v_*$,  and deleted edges $e_1=u_1u_1',\dots, e_k=u_ku_k'$. If $\mathcal{F}' = \underline{\mathbb{V}}/ \mathcal{S}'$, is constructed using $\mathcal{S}'$ where 
    \begin{align*}
        \mathcal{S}'(x) &= \mathcal{S}(x) \textup{ if }  x\in V,
    \end{align*}
    and
    \begin{equation*}
    \begin{array}{clr}
        \mathcal{S}'(f_{2j -1}) &= \ker(\alpha_{e_j})  &\textup{for $j \in\{1, \dots, k\}$}\\
        \mathcal{S}'(f_{2j}) &= \ker(\alpha_{e_j})  &\textup{for $j \in\{1, \dots, k\}$}\\
        \mathcal{S}'(f_{j}) &= \ker(\alpha_{j}) &\textup{for some $\alpha_{j} \in \mathbb{V}^{*}$ whenever $j \in\{2k+1, \dots, k + d\}$}
        \end{array}
    \end{equation*}
    such that $\alpha_{e_1}, \dots, \alpha_{e_k}, \alpha_{k+1},\dots,\alpha_{d}$ are linearly independent, and if $$\mathcal{S}'(v_*) = \bigcap_{i=1}^{k}  \ker(\alpha_{e_i}) \cap  \bigcap_{j=k+1}^{d}  \ker(\alpha_{j}),$$ then one has $h^{1}(\mathcal{F}') = 0$ and $\mathcal{F}'\in Z_{s, n}(\Gamma')$. 
\end{theorem}
\begin{proof}
 We will use the notation $\alpha_{i} = \alpha_{e_i}$ whenever $i\leq k$. Let us first note that defining $\mathcal{S}'(v_*) = \bigcap_{j=1}^{d}  \ker(\alpha_{j})$ yields a sheaf $\mathcal{F}' \in Z_{s, n}(\Gamma')$ since $\dim( \bigcap_{j=1}^{d} \ker(\alpha_{j})) = n - d = s$. We orient all edges towards the vertex $v_*$, and orient all edges $u_i'u_i$ from $u_i'$ towards $u_i$.

One has the long exact sequence from the short exact sequence \eqref{eq: les-ass}
\begin{equation*}
    \begin{array}{cl}
    0 \rightarrow & H^0(\Gamma', \mathcal{S}') \rightarrow H^0(\Gamma', \underline{\mathbb{V}}) \rightarrow H^0(\Gamma', \mathcal{F}') \\[10pt]
    &\xrightarrow{\delta} H^1(\Gamma', \mathcal{S}') \rightarrow H^1(\Gamma', \underline{\mathbb{V}}) \rightarrow H^1(\Gamma', \mathcal{F}') \rightarrow 0.
\end{array}
\end{equation*}
Thus, $H^{1}(\Gamma', \mathcal{F}') = 0$ if and only if $H^1(\Gamma', \mathcal{S}') \rightarrow H^1(\Gamma', \underline{\mathbb{V}})$ is surjective, and one has the same for the graph $\Gamma$. We will show that this map is surjective. We thus seek to find $ (w_{e})_{e \in E} \in H^1(\Gamma', \mathcal{S}')$ that map onto an arbitrarily given element of $H^{1}( \Gamma', \underline{\mathbb{V}})$. Up to $\im(d)$, elements of $H^{1}( \Gamma', \mathcal{S}')$ are given by formal sums $\sum_{e\in E(\Gamma')} s_e$, where $s_e\in \mathcal{S}'(e)$.
We will first treat two special cases, which we will then use together with the surjectivity of $H^{1}(\Gamma, \mathcal{S}) \rightarrow H^{1}(\Gamma, \underline{\mathbb{V}})$, to show that $H^{1}(\Gamma', \mathcal{S}') \rightarrow H^{1}(\Gamma', \underline{\mathbb{V}})$ is surjective.\\
We consider first cochains that are nonzero only on an added edge $f_{j}$ with $j > 2k$, and equal to $W$ on this edge. We denote the cochain by $W_{f_{j}}$. Since the elements $\alpha_{1}, \dots, \alpha_{k}, \alpha_{k+1} , \dots, \alpha_{d}$ are linearly independent, we have $\bigcap_{c \neq j-k} \ker(\alpha_c) + \ker(\alpha_{j-k}) = \mathbb{V}$, where $\alpha_{j-k}$ corresponds to the edge $f_j$. Thus there exist elements $w_{*} \in \bigcap_{c \neq j-k} \ker(\alpha_c)$ and a $w_{j-k} \in \ker(\alpha_{j-k})$ such that 
\begin{equation*}
    w_{*}  + w_{j-k} = W_{f_j}. 
\end{equation*}
  Now, consider the cochain defined by (see also \Cref{fig: proof_case_1}):
\begin{equation*}
\begin{array}{clcc}
    w_{f_{c}} &= - w_{*}  & \textup{if } c \in \{1, \dots, d+k\} \setminus \{j\} \\
    w_{f_{j}} &= w_{j-k}  & \\
    w_{e} &=0  &\textup{ otherwise.}
\end{array}
\end{equation*}
It is clear that this cochain defines an element of $H^{1}( \Gamma', \mathcal{S}')$. Moreover, one can see by using equation \eqref{eq: rule for h1} at $v_*$ with $w_{*}$, that it is equivalent to $W_{f_j}$, which is precisely what we needed to show.
\begin{figure}[h!]
\centering
\begin{tikzpicture}
\tikzstyle{vertex}=[circle, draw, fill=white, minimum size=16pt, inner sep=1pt, outer sep=0pt]
\tikzstyle{dot}=[fill=black, circle, inner sep=1.5pt]
\draw (0,0) ellipse (6cm and 3.5cm);
\node[vertex] (vstar) at (0,1.25)  {$v_*$};
\node[vertex] (v1) at (-3.5,2)  {$u_1$};
\node[vertex] (v2) at (-3.5,.5)  {$u_1'$};
\node[vertex] (v3) at (3.5,2)  {$u_2$};
\node[vertex] (v4) at (3.5,.5)  {$u_2'$};
\node[vertex] (v5) at (-1, -2)  {$v_1$};
\node[vertex] (v6) at (1, -2)  {$v_2$};

\draw[dashed]  (v1) -- (v2)  node[midway, left, draw=none, fill=none] {$e_1$};
\draw[dashed]  (v3) -- (v4)  node[midway, left, draw=none, fill=none] {$e_2$};
\draw[->] (v1) -- (vstar) node[midway, above, draw=none, fill=none] {$-w_*$};
\draw[->] (v2) -- (vstar) node[midway, below, draw=none, fill=none] {$-w_*$};
\draw[->] (v3) -- (vstar) node[midway, above, draw=none, fill=none] {$-w_*$};
\draw[->] (v4) -- (vstar) node[midway, below, draw=none, fill=none] {$-w_*$};
\draw[->] (v5) -- (vstar) node[midway, left, draw=none, fill=none] {$-w_*$};
\draw[->] (v6) -- (vstar) node[midway, right, draw=none, fill=none] {$w_4$};
\end{tikzpicture}
\caption{An illustration of the cochain for the first case for the $4$-dimensional $2$-extension. Edges not involved in the $2$-extension are not pictured. The edges $e_1$ and $e_2$ are replaced, and edges $f_1,\dots, f_6$ are added. The picture illustrates the construction of the cochain equal to $W_{f_{6}}$, where $f_6$ is the edge between $v_2$ and $v_*$.}
\label{fig: proof_case_1}
\end{figure}
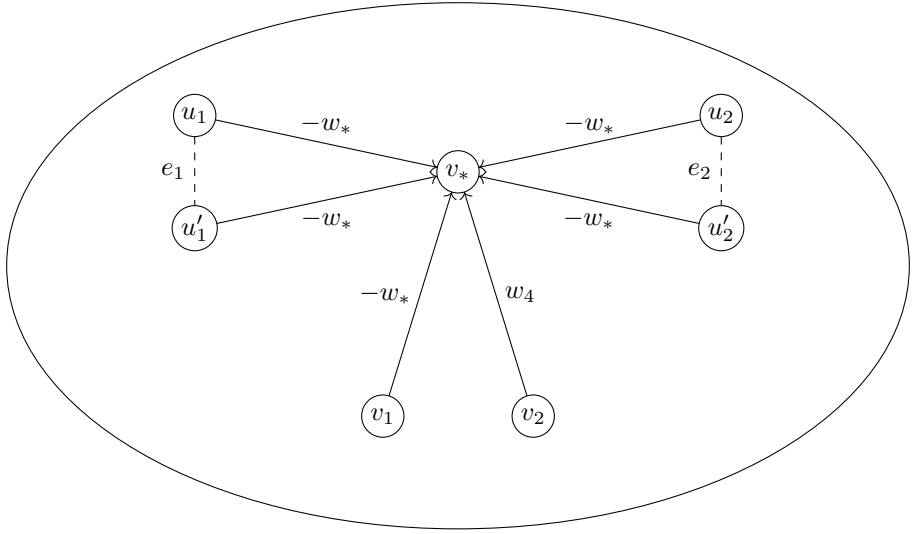

We now consider elements of the form $W_{f_{2j -1}} + W_{f_{2j}}$ with $j\leq k$, where $f_{2j-1},f_{2j}$ are given by $f_{2j-1} = v_*u_j, f_{2j} = v_*u_j'$, and they replace the edge $e_{j}$. By the linear independence assumption, we have $\bigcap_{c \neq j} \ker(\alpha_c) + \ker(\alpha_j) = \mathbb{V}$, and thus there exist elements $w_{*} \in \bigcap_{c \neq j} \ker(\alpha_c)$ and a $w_j \in \ker(\alpha_j)$ such that $w_* + w_j  = W$. One can then define a cochain
\begin{equation*}
\begin{array}{clcc}
    w_{f_{i}} &= - w_{*}  & \textup{if } i\in \{1, \dots, d+k\} \setminus \{2j, 2j-1\}\\
    w_{f_{2j-1}} &= w_{j}  & \\
    w_{f_{2j}} &= w_{j}  & \\
    w_{e} &=0  &\textup{ otherwise.}
\end{array}
\end{equation*}
 Again this cochain belongs to $H^{1}(\Gamma', \mathcal{S}')$, and using Equation \eqref{eq: rule for h1}, one can show that it maps onto $W_{f_{2j-1}} + W_{f_{2j}}$. See also the illustration in \Cref{fig: proof_case_2}.\\

\begin{figure}[h!]
\centering
\begin{tikzpicture}
\tikzstyle{vertex}=[circle, draw, fill=white, minimum size=16pt, inner sep=1pt, outer sep=0pt]
\tikzstyle{dot}=[fill=black, circle, inner sep=1.5pt]
\draw (0,0) ellipse (6cm and 3.5cm);
\node[vertex] (vstar) at (0,1.25)  {$v_*$};
\node[vertex] (v1) at (-3.5,2)  {$u_1$};
\node[vertex] (v2) at (-3.5,.5)  {$u_1'$};
\node[vertex] (v3) at (3.5,2)  {$u_2$};
\node[vertex] (v4) at (3.5,.5)  {$u_2'$};
\node[vertex] (v5) at (-1, -2)  {$v_1$};
\node[vertex] (v6) at (1, -2)  {$v_2$};

\draw[dashed]  (v1) -- (v2)  node[midway, left, draw=none, fill=none] {$e_1$};
\draw[dashed]  (v3) -- (v4)  node[midway, left, draw=none, fill=none] {$e_2$};
\draw[->] (v1) -- (vstar) node[midway, above, draw=none, fill=none] {$w_1$};
\draw[->] (v2) -- (vstar) node[midway, below, draw=none, fill=none] {$w_1$};
\draw[->] (v3) -- (vstar) node[midway, above, draw=none, fill=none] {$-w_*$};
\draw[->] (v4) -- (vstar) node[midway, below, draw=none, fill=none] {$-w_*$};
\draw[->] (v5) -- (vstar) node[midway, left, draw=none, fill=none] {$-w_*$};
\draw[->] (v6) -- (vstar) node[midway, right, draw=none, fill=none] {$-w_*$};
\end{tikzpicture}
\caption{An illustration of the cochain for the second case for the $4$-dimensional $2$-extension. The picture illustrates the construction of the cochain equal to $W_{f_1} + W_{f_2}$.}
\label{fig: proof_case_2}
\end{figure}
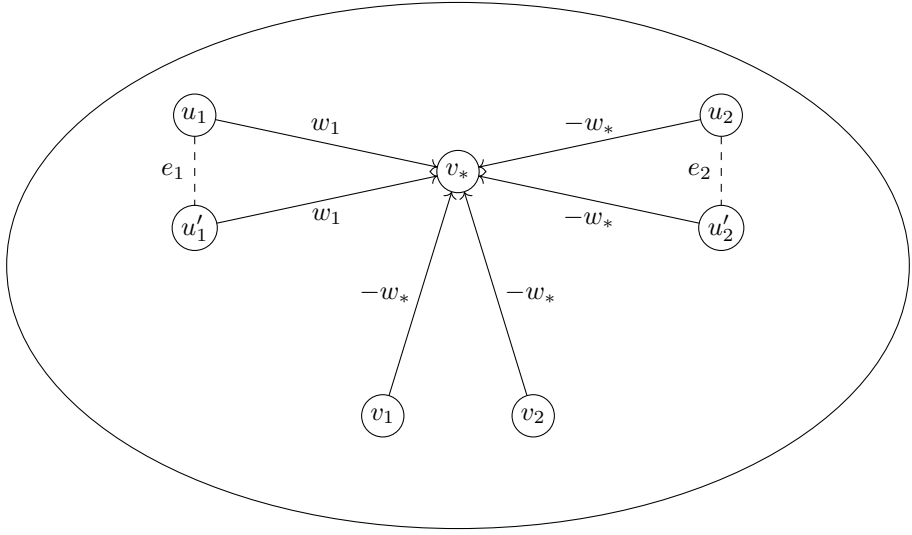
To complete the proof, we will use the independence of $\mathcal{F}$ on $\Gamma$ and the two cases that we have considered to show that the map is surjective in general. Let $N$ be the vector subspace of $H^{1}(\Gamma', \underline{\mathbb{V}})$ generated by 
\begin{equation*}
    W_{f_{i}}, W_{f_{2j-1}} + W_{f_{2j}},
\end{equation*}
  where $W\in \mathbb{V}$ and $j\in \{1, \dots, k\}$ and $i \in \{2k+1, \dots k + d\}$. From the above two cases, we know that
  \begin{equation*}
  N \subseteq \im(H^{1}(\Gamma', \mathcal{S}') \rightarrow H^{1}(\Gamma', \underline{\mathbb{V}})).
  \end{equation*}

Define a linear map
\begin{equation*}
\begin{array}{cl}
   \varphi: & \oplus_{e \in E(\Gamma)}  \mathcal{S}(e) \rightarrow \oplus_{e \in E(\Gamma')} \mathcal{S}'(e):\\
   
   &\begin{cases}
    W_{e} \mapsto W_{e} &\text{ if $e\in E(\Gamma) \cap E(\Gamma') $}\\
   W_{e_i} \mapsto W_{f_{2i - 1}} &\text{ if $e_i=u_iu_{i}'$ is a deleted edge},\\
   \end{cases}
\end{array}
\end{equation*}
and by extending linearly to all other vectors. This is clearly well defined. We define another linear map
\begin{equation*}
\begin{array}{cl}
   \psi: &H^{1}(\Gamma, \underline{\mathbb{V}}) \rightarrow H^{1}(\Gamma', \underline{\mathbb{V}})/N:\\
   &\begin{cases}
    W_{e} \mapsto W_{e} &\text{ if $e \in E(\Gamma) \cap E(\Gamma')$}\\
    W_{e} \mapsto W_{f_{2i -1}} &\text{ if $e_i=u_iu_{i}'$ is a deleted edge,}
    \end{cases}
\end{array}
\end{equation*}
and by extending linearly to all other vectors.
We need to check that $\psi$ is well-defined. By \eqref{eq: rule for h1}, it suffices to check that for each vertex $v$, if $\sum_{e\in N(v)} A_{e} = \sum_{e\in N^{+}(v)} (A_{e} + H)+ \sum_{e\in N^{-}(v)} (A_{e} - H)$ in $H^{1}(\Gamma, \underline{\mathbb{V}})$, then
\begin{equation*}
    \psi \left(\sum_{e\in N(v)} A_{e}\right) = \psi\left( \sum_{e\in N^{+}(v)} (A_{e} + H)+ \sum_{e\in N^{-}(v)} (A_{e} - H)\right)
\end{equation*}
in $H^{1}(\Gamma', \underline{\mathbb{V}})/N$. 
Thus, let $v$ be a vertex and suppose that $v$ is the terminal vertex of $e_{i_1}, \dots e_{i_{a}}$ and the initial vertex of $e_{i_{a+1}}, \dots e_{i_{b}}$ for some deleted edges $e_{i_1}, \dots e_{i_b}$, and that additionally, the edges $f_{j_1}, \dots, f_{j_c}$ are adjacent to $v$ (with $j_i \geq 2k$ for $i\in \{1, \dots, c\}$). Denote $S= \{e_{i_1}, \dots e_{i_{b}}, f_{j_1},\dots, f_{j_c} \}$. 
We have
\begin{align*}
\psi(\left(\sum_{e\in N_{\Gamma'}(v)} A_{e}\right) ) &= \sum_{p = 1}^{a} A_{f_{2  i_{p} - 1}} + \sum_{p = a+1}^{b} A_{f_{2  i_{p} - 1}} + \sum_{e\in N_{\Gamma'}^{+}(v) \setminus S}  A_{e} + \sum_{e\in N_{\Gamma'}^{-}(v) \setminus S}  A_{e} \\
&= \sum_{p = 1}^{a} (A_{f_{2  i_{p} - 1}} - H) - \sum_{p = a+1}^{b} (A_{f_{2  i_{p}}} +H) + \sum_{f_{j_q}}( - H)\\
&+ \sum_{e\in N_{\Gamma'}^{+}(v) \setminus S}  (A_{e} + H)+ \sum_{e\in N_{\Gamma'}^{-}(v) \setminus S}  (A_{e} -H)\\
&= \sum_{p = 1}^{a} (A_{f_{2  i_{p} - 1}} - H) + \sum_{p = a+1}^{b} (A_{f_{2  i_{p} -1}} + H) + \psi(\sum_{e\in N_{\Gamma'}(v) \setminus S}  (A_{e} \pm H)) \\
&=  \psi\left( \left(\sum_{e\in N_{\Gamma'}^{+}(v)} A_{e} + H\right) + \left(\sum_{e\in N_{\Gamma'}^{-}(v)} A_{e} - H\right)\right)
\end{align*}

We have used the first case above for the terms $H$ supported on $f_{j_{q}}$, since these belong to $N.$ We used $W_{f_{2j}} = -W_{f_{2j -1}} \textup{mod } N$, which follows from the second case, in the second and third equality. The map is thus well-defined. The map $\psi$ is surjective since, for any edge $e$, either $W_{e}$ obviously lies in the image, or $W_{e} \in N$, or $e=f_{2i}$ with $i\leq k$, in which case $W_{f_{2i}} = -W_{f_{2i -1}}$ clearly lies in the image.

Now note that the following diagram commutes:
\begin{figure}[h!]
    \centering
\begin{tikzcd}
\oplus_{e\in E(\Gamma)} \mathcal{S}(e) \arrow[r, "\varphi"] \arrow[d, "\textup{ mod } d"] & \oplus_{e \in E(\Gamma')} \mathcal{S}'(e) \arrow[d, " \textup{ mod } d"]  \\
H^{1}(\Gamma, \mathcal{S})  \arrow[d] & H^{1}(\Gamma', \mathcal{S}') \arrow[d] 
\\
{H^{1}(\Gamma, \underline{\mathbb{V}})} \arrow[r, "\psi"] & {H^{1}(\Gamma', \underline{\mathbb{V}})/N}
\end{tikzcd}
\end{figure}

The map $H^{1}(\Gamma, \mathcal{S})\rightarrow H^{1}(\Gamma, \underline{\mathbb{V}})$ is surjective by our assumption since $H^{1}(\Gamma, \mathcal{F}) = 0$ and by the long exact sequence associated to \eqref{eq: les-ass}. Thus, the composite map $\oplus_{e \in E} \mathcal{S}(e) \rightarrow H^{1}(\Gamma', \underline{\mathbb{V}})/N$ (going the bottom-left path in the diagram) is surjective as well since all maps involved are surjective. Thus, the map $H^{1}(\Gamma', \mathcal{S}') \rightarrow  H^{1}(\Gamma', \underline{\mathbb{V}})/ N$ must be surjective as well. By the two special cases, we see that $N$ is in the image of the map $H^{1}(\Gamma, \mathcal{S}) \rightarrow H^{1}(\Gamma', \underline{\mathbb{V}})$. Then, we see that the map $H^{1}(\Gamma', \mathcal{S}') \rightarrow H^{1}(\Gamma', \underline{\mathbb{V}})$ is surjective, completing the proof of the theorem.
\end{proof}

\subsection{Proof of \Cref{thm: main theorem}}
We need a lemma that states that one can choose the $\{\alpha_e\}$ to be linearly independent to apply the $k$-extension from \Cref{Extension-for-sheaves}.
\begin{lemma}\label{lem: dense-open -> general position}
    Let $s, n\in \mathbb{N}$ with $s <n/2$, and let $\Gamma=(V, E)$ be a multi-graph with at most $n - 2s$ edges between each pair of vertices. Let $\psi_{\Gamma}: Z_{s, n}(\Gamma) \rightarrow (\mathbb{V}^{*})^{|E|}: ((S_{v})_{v\in V}, \alpha_{e}) \mapsto (\alpha_{e})_{e\in E}$. For any nonempty Zariski-open subset $O\subseteq Z_{s,n}(\Gamma)$ and any subset $E' \subseteq E$ of size $|E'| \leq  n-s$ such that $|E' \cap N(v)| \leq n - 2s$, there exists $( \alpha_{e})_{e\in E} \in \psi_{\Gamma}(O)$ such that $\{\alpha_{e}\}_{e\in E'}$ is linearly independent. 
\end{lemma}
\begin{proof}
We note that $\{\alpha_{e}\}_{e\in E'}$ being linearly dependent defines a polynomial condition on points in $O$. Thus, if there exists any $x\in Z_{s, n}(\Gamma)$ such that the statement is true, then there exists an open subset $U \subseteq Z_{s,n}(\Gamma)$ such that the statement is true; and since $U\cap O$ is necessarily nonempty as $Z_{s, n}$ is irreducible, the statement follows. 

Let $E'=\{e_1, \dots e_{a}\}$, and suppose that $\alpha_{e_1},\dots, \alpha_{e_a}\in \mathbb{V}^{*}$ are linearly independent forms. It then suffices to show that we can pick $S_{v} \subseteq \bigcap_{e\in N(v) \cap E'}\ker(\alpha_{e})$ of dimension $s$ such that $S_{v} \cap S_{u} = \{0\}$ for all pairs $u,v\in V$, since this allows one to define a point $x\in Z_{s,n}(\Gamma)$ as desired. To ease notation, let us denote $D_v= \bigcap_{e\in N(v) \cap E'}\ker(\alpha_{e})$. Suppose by induction we have picked the coordinates for $S_{v_1}, \dots, S_{v_m}$. For $v_1$, one simply picks $S_{v_1}\subseteq D_{v_1}$ to be some arbitrary subspace of dimension $s$. Since $\dim(D_{v_1}) \geq n- (n-2s) = 2s$, this is possible. \\
Since $\dim(S_{v_i}) = s$, for all $i\leq m+1$ it holds that $\dim\left(S_{v_i} \cap D_{v_{m+1}}\right) \leq s$ and moreover one has $\dim\left(D_{v_{m+1}}\right) \geq n-(n-2s) =2s$. One can then apply \Cref{lem: Intersections_of_subspaces} in $\textup{Gr}(s , D_{v_{m+1}})$, after perhaps enlarging $S_{v_i} \cap D_{v_{m+1}}$ to be an $s$-dimensional subspace of $D_{v_{m+1}}$. Thus, for each $S_{v_i}$, there is a Zariski open subset of choices for $S_{v_{m+1}}\in \textup{Gr}(s, D_{v_{m+1}})$ such that $S_{v_{m+1}} \cap S_{v_i} = 0$. Intersecting these finitely many open subsets yields a non-empty open subset, as $\textup{Gr}(s, D_{v_{m+1}})$ is irreducible, and hence one can make a choice of $S_{v_{m+1}}$ such that the statement holds.
\end{proof}

In the following theorem, $K_{2}^{n-2}$ denotes the graph with two vertices and $n-2$ parallel edges. 
\begin{theorem}[Frank, Szegö \cite{frank2003constructive}]\label{thm: FrankSzegö}
Let $\Gamma=(V, E)$ be a $(n-1, n)$-tight multi graph. Then, $\Gamma$ can be constructed from $K_{2}^{n-2}$, using $(n-1)$-dimensional $k$-extensions, where $0\leq k\leq n-1$, such that no more than $(n-2)$ parallel edges are created at each pair of vertices. 
\end{theorem}

Szegö conjectured the following strengthening of \Cref{thm: FrankSzegö}.
\begin{conjecture}[Szegö \cite{SZEGO20061211}]\label{Szegö conjecture}
Let $\Gamma=(V, E)$ be a $(d, \ell)$-tight multi graph with $\ell < \frac{4d+2}{3}$. Then, $\Gamma$ can be constructed from $K_{2}^{2k-\ell}$, using $d$-dimensional $k$-extensions such that no more than $2d - \ell$ parallel edges are created at each pair of vertices. 
\end{conjecture}

\begin{theorem}\label{thm: main_thm_ass_sheaves}
    Let $n\geq 3$, and let $\Gamma$ be a $(n-1, n)-$sparse multigraph. Then $h^{1}(\Gamma, \mathcal{F}) = 0$ for a Zariski open subset of $\mathcal{F} \in Z_{1, n}(\Gamma)$. 
\end{theorem}
\begin{proof}
    It suffices to prove the theorem for $(n-1, n)$-tight graphs using \Cref{lem: Stresses extend}, and since one may always add edges to turn $\Gamma$ into a tight graph.
    The proof is by induction on the number of vertices. For the base case, we need to show that $K^{n-2}_{2}$ is independent. Pick $S_v$ and $S_{u}$ with $S_v\cap S_u = 0$. Whenever $\alpha_{e}$ are in general position, one can easily show $H^{1}(K^{n-1}_{2}, \mathcal{F}) = 0$ by using, for example, the long exact sequence arising from \eqref{eq: les-ass}.
    
       By \Cref{thm: FrankSzegö}, there exists some
     $(n-1)$-dimensional $k$-extension applied to a graph $\Gamma$ yielding a graph $\Gamma'$, where the removed edges are $e_1, \dots, e_{k}$. By the induction hypothesis, there is a Zariski open subset $O \subseteq Z_{1, n}(\Gamma)$ such that $h^{1}(\Gamma, \mathcal{F}) =0$ for all $\mathcal{F}\in O$. By \Cref{lem: dense-open -> general position}, we may find $\mathcal{F}\in O$ such that $\mathcal{F}(e_1) = \alpha_{e_1}, \dots, \mathcal{F}(e_k) = \alpha_{e_k}$ are linearly independent. Using \Cref{Extension-for-sheaves}, we may construct a sheaf $\mathcal{F}' \in Z_{1,n}(\Gamma')$ that satisfies $H^{1}(\Gamma', \mathcal{F'}) = 0$. Then, by \Cref{thm: Genericity_F} and using the fact that $Z_{1,n}(\Gamma')$ is irreducible, there is an open subset of $\mathcal{F}\in Z_{1, n}(\Gamma')$ such that $h^{1}(\Gamma, \mathcal{F}) = 0$, completing the proof of the theorem. 
\end{proof}

If \Cref{Szegö conjecture} is true, then \Cref{thm: main_thm_ass_sheaves} generalises to all sheaves $Z_{s, n}(\Gamma)$, where $s \leq \frac{n + 2}{4}$. We are now ready to prove the main theorem.

\begin{proof}[Proof of \Cref{thm: main theorem}]
One direction follows from \Cref{cor: necessary_concrete}. We only need to show that if $(n-2)\Gamma$ is $(n-1,n)$-sparse, then $H^{1}(I(\Gamma), \mathcal{M})=0$ for a Zariski open subset. 
By \Cref{thm: main_thm_ass_sheaves}, there is a Zariski open subset of $\mathcal{F} \in Z_{1, n}((n-2)\Gamma)$ with $h^{1}(\mathcal{F})=0$. Thus, picking any such $\mathcal{F}$, one has by \Cref{lem: Properties map Z -> M} that $\mathcal{M} = f(\mathcal{F})$ satisfies $h^{1}(\mathcal{M}) = 0$, where $f$ is the map from \Cref{lem: Properties map Z -> M}. Then, by \Cref{lem: Intersections_of_subspaces}, and \Cref{thm: Genericity_generalised_frameworks.} one has a Zariski open subset of $M_{1, n}(\Gamma)$ where $h^{1}(\mathcal{M}) = 0$. 
\end{proof}

\subsection{Inductive constructions for $M_{s, n}(\Gamma)$.}
In this section, we shall provide algebraic conditions that allow for inductive constructions for the class of sheaves $\mathcal{M}_{s, n}(\Gamma)$. In particular, one recovers $0$- and $1$-extensions for the classical cases (see \Cref{sec: inductive constructions gogs}). Since the proof method is essentially the same as that of \Cref{Extension-for-sheaves}, we shall omit some details in the proof.

\begin{theorem}\label{thm: k-extension}
Let $\Gamma$ be a graph, and let $\mathcal{M} = \mathbb{V}/\mathcal{S}$ be a sheaf from the family $M_{s,n}(\Gamma)$, and suppose that $H^{1}(I(\Gamma), \mathcal{M}) = 0$.
Suppose that $\Gamma'$ results from applying a $d$-dimensional $k$-extension, with new vertex $v_{*}$, new edges $f_{1} =u_{1}v_{*}, f_{2}=u_{1}'v_{*},\dots, f_{2k +1}=v_{1}v_*, \dots, f_{d+k}=v_{d-k}v_*$,  and deleted edges $e_1=u_1u_1',\dots, e_k=u_ku_k'$, such that no parallel edges are created. Then, let $\mathcal{S}'\in S_{s,n}(\Gamma')$ be a sheaf $I(\Gamma')$, such that
\begin{equation*}
\mathcal{S'}(x) =\mathcal{S}(x) \textup{ if $x\in V(I(\Gamma))$}.
\end{equation*}
Denote the resulting sheaf from $M_{s,n}$ by $\mathcal{M'}$.
Assume that for all $i \in \{1, \dots, k\}$:
\begin{equation}\label{eq: vertex v and ei}
    \mathcal{S'}(v_*)  + \mathcal{S}(u_i)=\mathcal{S'}(v_*)  + \mathcal{S}(u_i') = \mathcal{S'}(u_i)  + \mathcal{S}(u_i')
\end{equation}

Furthermore, assume that for $j \in \{1, \dots, d-k\}$ one has
\begin{equation} \label{eq: Algebraic_condition_ext1}
\left( \left( \bigcap_{i \neq j,\\ i=1}^{d-k}\mathcal{S}'(v_i) + \mathcal{S}'(v_*)\right) \cap \left( \bigcap_{i=1}^{k}\mathcal{S}'(u_i) + \mathcal{S}'(u_i') \right)\right) + \mathcal{S}'(v_j)  = \mathbb{V},
\end{equation}
and for all $j \in \{1, \dots, k\}$ one has
\begin{equation} \label{eq: Algebraic_condition_ext2}
\left( \left( \bigcap_{i=1}^{d-k}\mathcal{S}'(v_i) + \mathcal{S}'(v_*)\right) \cap \left( \bigcap_{i=1, j \neq i}^{k}\mathcal{S}'(u_i) + \mathcal{S}'(u_i') \right)\right) + \mathcal{S}'(u_j)  = \mathbb{V}.
\end{equation}
Then $H^{1}(I(\Gamma'), \mathcal{M}')  = 0$. In the case where the intersection is over an empty index set, the intersection is understood to be $\mathbb{V}$.
\end{theorem}
\begin{proof}
    The idea behind the proof is essentially the same as that of \Cref{Extension-for-sheaves}. One may orient all edges $e\sim v \in I(\Gamma)$ such that $[e: e\sim v]=1,[v: e\sim v]=-1$ 
    First we show that for two special cases, certain cochains are in the image of the map $H^{1}(I(\Gamma'), \mathcal{S}') \rightarrow H^{1}(I(\Gamma'), \underline{\mathbb{V}})$. We let $W_{v\sim e} \in H^{1}(I(\Gamma'), \underline{\mathbb{V}})$ be defined for $W\in \mathbb{V}$ and $v \sim e \in I$, as
\begin{equation*}
\begin{cases}
    W_{v\sim e}(v\sim e) = W \\
    W_{v\sim e}(x\sim e') = 0 &\text{for other incidences $x \sim e' \in I$.} 
\end{cases}
\end{equation*}
The first case is given by elements of the form $W_{v_* \sim f_{2k + j}}$ for some $j \in \{1, \dots, d-k\}$. The second case is given by elements of the form $W_{v_* \sim f_{2j-1}} + W_{v_* \sim f_{2j}}$ for $j \leq k$. 
 For the first case, by \eqref{eq: Algebraic_condition_ext1}, one can find elements
\begin{align*}
    &w_r \in \left( \left( \bigcap_{i=1, i \neq j}^{d-k}\mathcal{S}'(v_{i}) + \mathcal{S}'(v_*)\right) \cap \left( \bigcap_{i=1}^{k}\mathcal{S}'(u_i) + \mathcal{S}'(u_i') \right)\right) \\
    &w_{v_{j}} \in \mathcal{S}'(v_{j}) 
\end{align*}

such that $w_r + w_{v_{j}}  = W$. For each $u_i$ we can find $w_{u_i}\in \mathcal{S}'(u_i)$ and $w_{v_*, u_i}\in \mathcal{S}'(v_*)$ such that $w_{u_i} +w_{v_*, u_i} = w_{r}$, and similarly for each $u_i'$. For each $v_i$, one can find $w_{v_i}\in \mathcal{S}'(v_i)$ and $w_{v_*, v_i}\in \mathcal{S}'(v_*)$ such that $w_{v_i} + w_{v_*, v_i} = w_r$. One can easily use these elements to define a cochain belonging to $H^{1}(I(\Gamma'),\mathcal{S}')$ (see \Cref{fig: proof_case_1_bis}) mapping onto $W_{v_* \sim f_{2k + j}} \in H^{1}(I(\Gamma'), \underline{\mathbb{V}})$.

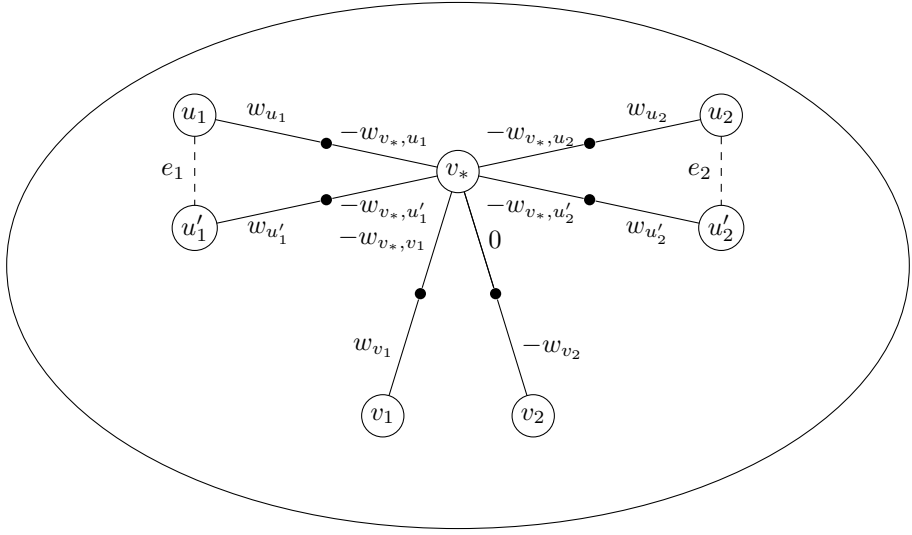
\begin{figure}[h!]
\centering
\begin{tikzpicture}
\tikzstyle{vertex}=[circle, draw, fill=white, minimum size=16pt, inner sep=1pt, outer sep=0pt]
\tikzstyle{dot}=[fill=black, circle, inner sep=1.5pt]
\draw (0,0) ellipse (6cm and 3.5cm);
\node[vertex] (vstar) at (0,1.25)  {$v_*$};
\node[dot] (vstar* - v1) at (-1.75, 1.625)  {};
\node[vertex] (v1) at (-3.5,2)  {$u_1$};
\node[dot] (vstar* - v2) at (-1.75, 0.875)  {};
\node[vertex] (v2) at (-3.5,.5)  {$u_1'$};
\node[dot] (vstar* - v3) at (1.75, 1.625)  {};
\node[vertex] (v3) at (3.5,2)  {$u_2$};
\node[dot] (vstar* - v4) at (1.75, 0.875)  {};
\node[vertex] (v4) at (3.5,.5)  {$u_2'$};
\node[dot] (vstar* - v5) at (-0.5, -0.375)  {};
\node[vertex] (v5) at (-1, -2)  {$v_1$};
\node[dot] (vstar* - v6) at (0.5, -0.375)  {};
\node[vertex] (v6) at (1, -2)  {$v_2$};
\draw[dashed]  (v1) -- (v2)  node[midway, left, draw=none, fill=none] {$e_1$};
\draw[dashed]  (v3) -- (v4)  node[midway, left, draw=none, fill=none] {$e_2$};

\draw (vstar) -- (vstar* - v1) node[midway, above, draw=none, fill=none] {$-w_{v_*, u_1}$};
\draw (vstar* - v1) -- (v1) node[midway, above, draw=none, fill=none] {$w_{u_1}$};
\draw (vstar) -- (vstar* - v2) node[midway, below, draw=none, fill=none] {$-w_{v_*, u_1'}$};
\draw (vstar* - v2) -- (v2) node[midway, below, draw=none, fill=none] {$w_{u_1'}$};
\draw (vstar) -- (vstar* - v3) node[midway, above, draw=none, fill=none] {$-w_{v_*, u_2}$};
\draw (vstar* - v3) -- (v3) node[midway, above, draw=none, fill=none] {$w_{u_2}$};
\draw (vstar) -- (vstar* - v4) node[midway, below, draw=none, fill=none] {$-w_{v_*, u_2'}$};
\draw (vstar* - v4) -- (v4) node[midway, below, draw=none, fill=none] {$w_{u_2'}$};

\draw (vstar) -- (vstar* - v6) node[midway, right, draw=none, fill=none] {$0$};

\draw (vstar) -- (vstar* - v5) node[midway, left, draw=none, fill=none] {$-w_{v_* ,v_1}$};
\draw (vstar* - v5) -- (v5) node[midway, left, draw=none, fill=none] {$w_{v_1}$};
\draw (vstar* - v6) -- (v6) node[midway, right, draw=none, fill=none] {$-w_{v_2}$};
\draw  (vstar) -- (vstar* - v6)  node[midway, left, draw=none, fill=none] {};
\end{tikzpicture}
\caption{An illustration of the cochain for the first case for the $4$-dimensional $2$-extension. Edges not involved in the $2$-extension are not pictured. The edges $e_1$ and $e_2$ are replaced, and edges $f_1,\dots, f_6$ are added. The picture illustrates the construction of the cochain equal to $W_{v_* \sim f_6}$, where $f_6$ is the edge between $v_2$ and $v_*$.}
\label{fig: proof_case_1_bis}
\end{figure}

For the second case, by
\eqref{eq: Algebraic_condition_ext2}, one can find elements
\begin{align*}
    &w_r \in \left( \left( \bigcap_{i=1}^{d-k}\mathcal{S}'(v_i) + \mathcal{S}'(v_*)\right) \cap \left( \bigcap_{i=1, i \neq j}^{k}\mathcal{S}'(u_i) + \mathcal{S}'(u_i') \right)\right)\\
    &w_{u_j} \in \mathcal{S}'(u_j) 
\end{align*}
such that $w_r + w_{u_j}  = W$. 
By the assumptions, one can write $w_{u_j} = w_{u_j'}  + w_{v_*, u_j'}$, where $w_{u_j'}\in \mathcal{S}'(u_j')$ and $w_{v_*, u_j'}\in \mathcal{S}'(v_*)$. For all other edges, one can write $w_{r} = w_{v_i} + w_{v_*, v_i}$ with $w_{v_i}\in \mathcal{S}'(v_j),w_{v_*, v_i}\in \mathcal{S}'(v_*)$ or $w_{r} = w_{u_i} + w_{v_*, u_i}$ with $w_{u_i}\in \mathcal{S}'(u_i),w_{v_*, u_i}\in \mathcal{S}'(v_*)$ and similarly for $u_i'$. One can then use these elements to define a cochain mapping onto $W_{v_* \sim f_{2j-1}} + W_{v_* \sim f_{2j}}$ (see also \Cref{fig: proof_case_2_bis}).
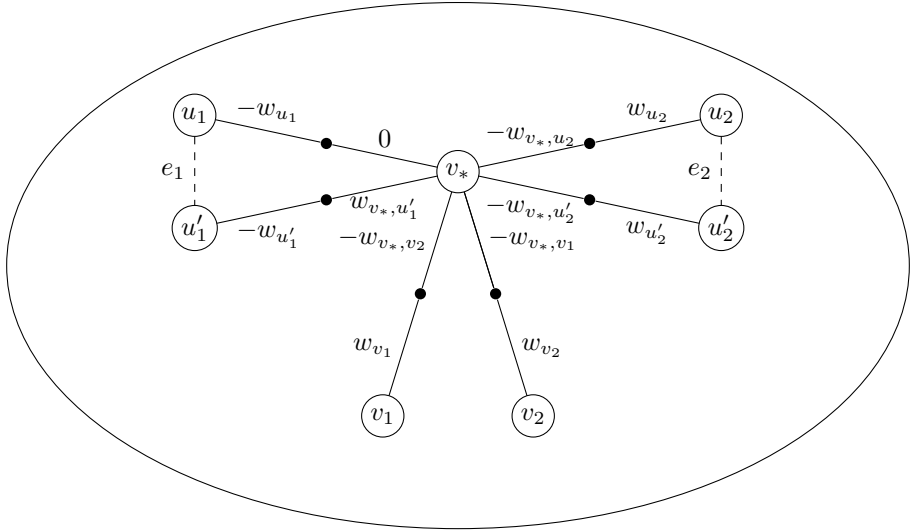
\begin{figure}[h!]
\centering
\begin{tikzpicture}
\tikzstyle{vertex}=[circle, draw, fill=white, minimum size=16pt, inner sep=1pt, outer sep=0pt]
\tikzstyle{dot}=[fill=black, circle, inner sep=1.5pt]
\draw (0,0) ellipse (6cm and 3.5cm);
\node[vertex] (vstar) at (0,1.25)  {$v_*$};
\node[dot] (vstar* - v1) at (-1.75, 1.625)  {};
\node[vertex] (v1) at (-3.5,2)  {$u_1$};
\node[dot] (vstar* - v2) at (-1.75, 0.875)  {};
\node[vertex] (v2) at (-3.5,.5)  {$u_1'$};
\node[dot] (vstar* - v3) at (1.75, 1.625)  {};
\node[vertex] (v3) at (3.5,2)  {$u_2$};
\node[dot] (vstar* - v4) at (1.75, 0.875)  {};
\node[vertex] (v4) at (3.5,.5)  {$u_2'$};
\node[dot] (vstar* - v5) at (-0.5, -0.375)  {};
\node[vertex] (v5) at (-1, -2)  {$v_1$};
\node[dot] (vstar* - v6) at (0.5, -0.375)  {};
\node[vertex] (v6) at (1, -2)  {$v_2$};

\draw[dashed]  (v1) -- (v2)  node[midway, left, draw=none, fill=none] {$e_1$};
\draw[dashed]  (v3) -- (v4)  node[midway, left, draw=none, fill=none] {$e_2$};

\draw (vstar) -- (vstar* - v1) node[midway, above, draw=none, fill=none] {$0$};
\draw (vstar* - v1) -- (v1) node[midway, above, draw=none, fill=none] {$-w_{u_1}$};
\draw (vstar) -- (vstar* - v2) node[midway, below, draw=none, fill=none] {$w_{v_*, {u_1'}}$};
\draw (vstar* - v2) -- (v2) node[midway, below, draw=none, fill=none] {$-w_{u_1'}$};
\draw (vstar) -- (vstar* - v3) node[midway, above, draw=none, fill=none] {$-w_{v_*, {u_2}}$};
\draw (vstar* - v3) -- (v3) node[midway, above, draw=none, fill=none] {$w_{u_2}$};
\draw (vstar) -- (vstar* - v4) node[midway, below, draw=none, fill=none] {$-w_{v_*, {u_2'}}$};
\draw (vstar* - v4) -- (v4) node[midway, below, draw=none, fill=none] {$w_{u_2'}$};

\draw (vstar) -- (vstar* - v6) node[midway, right, draw=none, fill=none] {$-w_{v_* ,v_1}$};

\draw (vstar) -- (vstar* - v5) node[midway, left, draw=none, fill=none] {$-w_{v_* ,v_2}$};
\draw (vstar* - v5) -- (v5) node[midway, left, draw=none, fill=none] {$w_{v_1}$};
\draw (vstar* - v6) -- (v6) node[midway, right, draw=none, fill=none] {$w_{v_2}$};
\draw  (vstar) -- (vstar* - v6)  node[midway, left, draw=none, fill=none] {};
\end{tikzpicture}
\caption{An illustration of the cochain for the second case for the $4$-dimensional $2$-extension. The picture illustrates the construction of the cochain equal to $W_{u_1, u_1 \sim f_1} + W_{u_1', u_1' \sim f_2} $}
\label{fig: proof_case_2_bis}
\end{figure}

Let $N$ be the vector subspace of $H^{1}(I(\Gamma'), \underline{\mathbb{V}})$ generated by 
\begin{equation*}
    W_{v_{*} \sim f_{2k+i}}, W_{v_* \sim f_{2i -1}} + W_{v_* \sim f_{2i}},
\end{equation*}
  where $W\in \mathbb{V}$ and $i\in \{1, \dots d-k\}$ or $i\in \{1, \dots, k\}$.

One can define 
\begin{equation*}
\begin{array}{cl}
   \varphi: & \oplus_{v\sim e\in I(\Gamma)}  \mathcal{S}(v) \rightarrow \oplus_{v\sim e\in I(\Gamma')} \mathcal{S}'(v):\\
   &\begin{cases}
    w_{v \sim e} \mapsto w_{v \sim e} &\text{ if $v$ and $e$ are present in $\Gamma'$}\\
    w_{v_i \sim e_i} \mapsto w_{ v_i \sim f_{2i-1}} &\text{ if $e_i=v_iv_{i}'$ is a deleted edge}\\
    w_{v_{i}' \sim e_i} \mapsto w_{ v_{i}' \sim f_{2i}} &\text{ if $e_i=v_iv_{i}'$ is a deleted edge}
   \end{cases}
\end{array}
\end{equation*}
This is clearly well defined. We define
\begin{equation*}
\begin{array}{cl}
   \psi: & H^{1}(I(\Gamma), \underline{\mathbb{V}}) \rightarrow H^{1}(I(\Gamma'), \underline{\mathbb{V}})/N:\\
   &\begin{cases}
    W_{v \sim e} \mapsto W_{v \sim e} &\text{ if $v$ and $e$ are present in $\Gamma'$}\\
    W_{v_i \sim e_i} \mapsto W_{v_i \sim f_{2i-1}} &\text{ if $e_i=v_iv_{i}'$ is a deleted edge}\\
    W_{v_{i}' \sim e_i} \mapsto W_{v_{i}'\sim f_{2i}} &\text{ if $e_i=v_iv_{i}'$ is a deleted edge}
   \end{cases}
\end{array}
\end{equation*}

One can show well-defined-ness of $\psi$ using the first and second cases. The remainder of the proof follows verbatim as in the proof of \Cref{Extension-for-sheaves}.
\end{proof}

\section{Applications to graph-of-groups realisations}\label{sec: applications to gogs}
\subsection{Graph-of-groups realisations of hypergraphs}\label{subsec: Background - graphs of groups}
We describe here the approach to rigidity theory developed in \cite{graphsofgroups}.

Let $G$ be a group and let $S(G)$ be the lattice of subgroups of $G$.
Given a hypergraph $\Gamma = (V,E)$, a realisation of $\Gamma$ as a graph of groups, (or a graph-of-groups realisation) is a function
\begin{equation*}
    \begin{array}{rccl}
    \rho:& V\cup E &\rightarrow &S(G)\\ 
    &x&\mapsto &\rho(x).
    \end{array}
\end{equation*}

By taking the incidence graph of $\Gamma$ and placing the group $\rho(v)$ on the vertex representing the vertex $v\in V$, $\rho(e)$ on the vertex representing the edge $e\in E$, and $\rho(v \sim e)=\rho(v)\cap \rho(e)$, we obtain a so-called graph of groups. The monomorphisms 
\begin{align*}
    i_v:&\rho(v \sim e) \rightarrow \rho(v)\\
    i_e:&\rho(v \sim e) \rightarrow \rho(e),
\end{align*}
are defined by inclusion. We will use the notation $\rho(\Gamma)$ to denote a hypergraph together with a realisation $\rho$ as a graph of groups. When $\rho(x)$ are Lie subgroups for all $x\in V\cup E\cup I$, we will say that $\rho(\Gamma)$ is a Lie graph-of-groups realisation.

\begin{definition}
\label{def:motion}
Let $I = E(I(\Gamma))$ be the set of incidences of a hypergraph $\Gamma$. Define a {\em motion} of a realisation $\rho(\Gamma)$ to be an indexed set of group elements $(\sigma_x)_{x\in V\cup E \cup I} \in G^{\vert V\vert + \vert E\vert +\vert I \vert}$, such that for every $x\in V\cup E$ and $i\in I$ with $i=x\sim y$ or $i=y \sim x$ one has that 
\begin{equation}
\label{eq:incidencepreserving}
    \sigma_i^{-1}\sigma_x\in \rho(x).
\end{equation}
\end{definition}
        
A motion $M=(\sigma_x)_{x\in V\cup E \cup I}$ of a realisation $\rho(\Gamma)$ defines a new realisation $\rho'(\Gamma)$, obtained by conjugating the groups $\rho(v)$ and $\rho(e)$ with the corresponding group elements in the motion. 
In other words, if $\rho$ is a realisation, then  $\rho^{M}$ is defined as
$$\begin{array}{rccl}
    \rho^{M}:& \Gamma &\rightarrow &S(G)\\
    &v&\mapsto &\rho(v)^{\sigma_v} \text{ for vertices}, \\ 
    &e&\mapsto &\rho(e)^{\sigma_e} \text{ for edges}. \\ 
\end{array}$$

When $\rho(\Gamma)$ is a Lie graph-of-groups-realisation, one gets a definition of infinitesimal motion in this context. Denote the Lie algebra of $\rho(x)$ by $\mathfrak{h}_x$, and the Lie algebra of $G$ by $\mathfrak{g}$. An infinitesimal motion is defined to be an element $$w_{v_1}, \cdots, w_{v_n}, w_{e_1}, \cdots w_{e_k}, w_{i_1} \cdots w_{i_m} .$$ of
\begin{equation*}
    \mathfrak{g}/\mathfrak{h}_{v_1} \times \cdots \times \mathfrak{g}/\mathfrak{h}_{v_n} \times \mathfrak{g}/\mathfrak{h}_{e_1} \times \cdots \times \mathfrak{g}/\mathfrak{h}_{e_k} \times  \mathfrak{g}/\mathfrak{h}_{i_1} \times \cdots \times  \mathfrak{g}/\mathfrak{h}_{i_m} .
\end{equation*}
such that for any incidence $i=v \sim e$, one has
\begin{align*}
    w_i - w_v &= 0 \mod \mathfrak{h}_v\\
    w_i - w_e &= 0 \mod \mathfrak{h}_e
\end{align*}

We denote the resulting vector space by $IM_\rho$. It can be shown that this is isomorphic to the usual definition \cite[Theorem 4.8]{graphsofgroups} in the case of the Euclidean group. One also has so-called trivial infinitesimal motions. One defines $i: \mathfrak{g}\rightarrow IM_\rho:  w\mapsto [(w, w, \dots, w)]$. The space which is the image of $i$ is the space of trivial infinitesimal motions. The following lemma states that when frameworks are 'large enough', one has that all Lie-algebra elements define motions. 

\begin{lemma}\label{lem: Trivial_motions}
Let $\rho(\Gamma)$ be a Lie graph-of-groups realisation. Then, with $$i: \mathfrak{g}\rightarrow IM_\rho:  w\mapsto [(w, \dots, w)]$$ one has $$\ker(i) = \bigcap_{x\in V(\Gamma)\cup E(\Gamma)} \mathfrak{h}_{x}.$$
\end{lemma}

The following bound was proved in \cite{graphsofgroups}:
\begin{theorem}\label{thm: weak_maxwell}
Let $\rho(\Gamma)$ be a Lie graph-of-groups realisation in a group $G$. Then 
\begin{align*}
\dim(IM_{\rho}) \geq& \sum_{x\in V\cup E} \dim(G/ \rho(x)) +\sum_{i\in I} \dim(G/ \rho(i)) \\&- \sum_{i=x*y \in I}(\dim(G/ \rho(x)) +\dim(G/ \rho(y))) 
\end{align*}
\end{theorem}
This was then used to show that the graph-of-groups realisation satisfy a certain sparsity condition, which gives the Maxwell condition in the case of the Euclidean group. 

In the remainder of this paper, we will treat in detail the example which corresponds to bar-joint frameworks, though the model does treat general situations.
\begin{example}\label{ex: se(d)}
    Let $(\Gamma, p)$ be a bar-joint framework, i.e. $\Gamma$ is a graph and $p: V(\Gamma) \rightarrow \mathbb{R}^{d}$ is simply a function. We define a Lie-graph-of-groups realisation $\rho_{p}(\Gamma)$ in the group $$E(d) = \{\begin{bmatrix}
        A & b \\
        0 &1
    \end{bmatrix}\in \textup{Gl}(\mathbb{R}, d+1) ~\vert~ A^{t}=A^{-1}, b\in \mathbb{R}^{d}\}$$ by
    \begin{equation*}
    \begin{array}{rccl}
    \rho_p:& V\cup E &\rightarrow &S(E(d))\\ 
    &v&\mapsto &\text{Stab}(p(v)).\\
    &e=vw&\mapsto &\text{Stab}(p(v)) \cap \text{Stab}(p(w))
    \end{array}
    \end{equation*}
    For any point $q \in \mathbb{R}^{d}$, $\Stab(q)$ is the set of matrices of the form
    \begin{equation*}
    \begin{bmatrix}
        A & q - A q\\
        0 & 1
    \end{bmatrix}
    \end{equation*}
    The intersection of two such stabiliser subgroups, say of $q_1, q_2\in \mathbb{R}^{d}$ with $q_1 \neq q_2$, is the set of such matrices such that additionally $A\cdot (q_1 - q_2)= q_1 - q_2$. Geometrically, one can view these elements as the rotations whose axis of rotation contains both $q_1$ and $q_2$. 
    The Lie algebra of $E(d)$, denoted by $\mathfrak{e}(d)$, consists of the matrices 
    \begin{equation*}
    \begin{bmatrix}
    S & t\\
    0 & 0
    \end{bmatrix}\in \mathbb{R}^{(d+1) \times (d+1)}
    \end{equation*}
    where $S\in \mathbb{R}^{d \times d}$, and $S^{t} = - S$, and $t\in \mathbb{R}^{d}$, and the Lie algebra of  $\Stab(q)$ for $q\in \mathbb{R}^d$, consists of the matrices
    \begin{equation*}
    \begin{bmatrix}
    S & -Sq \\
    0 & 0
    \end{bmatrix}
    \end{equation*}
    where $S\in \mathbb{R}^{d \times d}$, and $S^{t} = - S$. The dimension of the stabiliser of any point $q\in \mathbb{R}^{d}$ is seen to be $\binom{d}{2}$, since the Lie algebra clearly has this dimension. Finding the Lie algebra of the intersection of two such stabilisers (say of $q_1$ and $q_2$, with $q_1 \neq q_2$), amounts to solving $-Sq_1 = -Sq_2$, which is equivalent to the vector $q_1 - q_2$ being a zero of $S$. By conjugating one can assume that $q_1= 0$ and $q_2 = (1, 0, \cdots, 0)$, and one sees that the dimension of $\rho(e)$ is $\binom{d-1}{2}$. One also has that the intersection of the Lie algebras of stabilisers of $k$ points is trivial if the $k$ points affinely span a subspace of dimension at least $d-1$.

    We note for what follows that, if $\mathfrak{h}_v$ is the Lie algebra of a point stabiliser of a point and if $\mathfrak{h}_e$ is the Lie algebra of the stabiliser of two points, that
    \begin{align*}
        \dim(\mathfrak{g}/\mathfrak{h}_v) &= \binom{d+1}{2} - \binom{d}{2} =d\\
        \dim(\mathfrak{g}/\mathfrak{h}_e) &= \binom{d+1}{2} - \binom{d - 1}{2} =2d -1.
    \end{align*}
    
    Let us now look at the particular cases $d=2$ and $d=3$. When $d=2$, the Lie algebra $\mathfrak{e}(2)$ consists of the $3$-dimensional vector space
    \begin{equation*}
    \begin{bmatrix}
     0 & -\omega & a \\
    \omega & 0  & b\\
    0 & 0  & 0
    \end{bmatrix}
    \end{equation*}
    where $a,b,\omega \in \mathbb{R}$ and the Lie algebra of a stabiliser of a point $(x, y) \in \mathbb{R}^{2}$ is the $1$-dimensional vector space of matrices of the form
    \begin{equation*}
    \begin{bmatrix}
     0 &  -\omega &  \omega y \\
    \omega & 0  & -\omega x\\
    0 & 0  & 0
    \end{bmatrix} 
     \end{equation*}
    where $ -\omega \in \mathbb{R}.$ The Lie algebra of the stabilisers of two points is $\{0\}$.

    For $d=3$, the Lie algebra $\mathfrak{e}(3)$ consists of the $6$-dimensional vector space
    \begin{equation*}
    \begin{bmatrix}
     0 & -\omega_3 &\omega_2& a \\
    \omega_3 & 0  & -\omega_1& b\\
    -\omega_2 & \omega_1  & 0  & c\\
    0 & 0 & 0 & 0
    \end{bmatrix}
    \end{equation*}
    where $\omega_1, \omega_2, \omega_3, a,b,c\in \mathbb{R}$ and the Lie algebra of a stabiliser of a point $(x, y, z) \in \mathbb{R}^{3}$ is the $3$-dimensional vector space of matrices of the form
    \begin{equation*}
    \begin{bmatrix}
     0 & -\omega_3 &\omega_2& \omega_3 y -\omega_2 z\\
    \omega_3 & 0  & -\omega_1& - \omega_3 x +\omega_1 z\\
    -\omega_2 & \omega_1  & 0  & \omega_2 x - \omega_1 y\\
    0 & 0 & 0 & 0
    \end{bmatrix} 
     \end{equation*}
    where $\omega_1, \omega_2, \omega_3 \in \mathbb{R}.$ We note that, just keeping the coordinates $(\omega_1, \omega_2, \omega_3, a,b,c) \in \mathbb{R} ^{6}$, the Lie algebra of the stabiliser consists precisely of $$\{(\omega , -\omega \times (x,y,z)) \in \mathbb{R}^{6}~\vert~ \omega \in \mathbb{R}^{3}\}.$$
    With this notation, one can see that the intersection of two stabilisers, say again of $q_1=(x_1,y_1, z_1)\in \mathbb{R}^{3}$ and $q_2=(x_2, y_2, z_2) \in \mathbb{R}^{3}$ with $q_1\neq q_2$ in $\mathfrak{e}(3)$, is the $1$-dimensional vector space $\langle (q_1 - q_2, q_2 \times q_1 )\rangle$. In matrix form, these are the real multiples of
        \begin{equation*}
    \begin{bmatrix}
     0 & -(z_1 - z_2) &(y_1- y_2)& y_2 z_1 - y_1 z_2 \\
    (z_1 - z_2)  & 0  & -(x_1 - x_2)& x_1 z_2 - x_2 z_1 \\
    -(y_1- y_2) & (x_1 - x_2)  & 0  &  x_2 y_1- x_1 y_2\\
    0 & 0 & 0 & 0
    \end{bmatrix} 
     \end{equation*}

\end{example}

\subsection{Motion sheaves from graphs of groups}\label{sec: basics}

Let $G$ be a Lie group. Throughout this section, we fix a Lie graph-of-groups realisation $\rho: V\cup E\rightarrow S(G): x\mapsto \rho(x)$, such that for each hyperedge $e$, one has $\rho(e) = \cap_{v\in e} \rho(v)$. This simplifies the description of the vector space of infinitesimal motions. Indeed, it is easy to show that the resulting vector space of infinitesimal motions of $\rho$ is isomorphic to the vector space 
\begin{align*}
     (\overline{w_{v_1}}, \dots, \overline{w_{e_{|E|}}} )  &\in \mathfrak{g}/\mathfrak{h}_{v_1}  \times \dots \times \mathfrak{g}/\mathfrak{h}_{v_{|V|}} \times \mathfrak{g}/\mathfrak{h}_{e_1}  \times \dots \times \mathfrak{g}/\mathfrak{h}_{e_{|E|}},  \\
    &\textup{ such that } w_{e} - w_{v} \in \mathfrak{h}_v \text{ if }v \in e. 
\end{align*}
 The assumption that $\rho(e) = \bigcap_{v\in V} \rho(v)$ holds in many  graph-of-groups realisations arising from combinatorial-geometric problems, with a notable exception being projective rigidity. The results in the current subsection are, in fact, easily adaptable to more general Lie-graph-of-groups realisations (i.e., those where $\rho(e) \neq \bigcap_{v\in V} \rho(v)$), but the results from the next section do not. 
 
We define the \textit{motion sheaf} $\mathcal{M}$ associated to $\rho$. Let
\begin{align*}
    \mathcal{M}(x) &= \mathfrak{g}/\mathfrak{h}_x &\text{ for } x\in V(I(\Gamma)).\\
    \mathcal{M}(v\sim e) &= \mathfrak{g}/\mathfrak{h}_v &\text{ for } v\sim e \in E(I(\Gamma)).
\end{align*}
Where let $$r^{e}_{v \sim e}: \mathfrak{g}/\mathfrak{h}_e \rightarrow \mathfrak{g}/\mathfrak{h}_{v}: w \mod \mathfrak{h}_e \mapsto w \mod \mathfrak{h}_v$$ and 
we let $r^{v}_{v \sim e}: \mathfrak{g}/\mathfrak{h}_v \rightarrow \mathfrak{g}/\mathfrak{h}_v$ be the identity. We see that the $0$th cohomology group is
\begin{align*}
    H^{0}(I(\Gamma), \mathcal{M}) = \{ (w_{x})_{x\in V\cup E} \in \bigoplus_{x\in V\cup E} \mathfrak{g}/\mathfrak{h}_x ~\vert ~ w_e - w_{v} = 0 \mod  \mathfrak{h}_v ~\forall  v\sim e \in I\}
\end{align*}
It is easily verified that $H^{0}(I(\Gamma), \mathcal{M})$ is isomorphic to the space of infinitesimal motions. By \Cref{thm: Maxwell-rule}, we have
\begin{align*}
    \dim(H^{0}(I(\Gamma), \mathcal{M})) - \dim(H^{1}(I(\Gamma), \mathcal{M})) &= \sum_{x\in V\cup E} \dim(\mathfrak{g}/\mathfrak{h}_{x})-\sum_{v\sim e \in I} \dim(\mathfrak{g}/\mathfrak{h}_{v}-\\
     &= \sum_{x\in V} \dim(\mathfrak{g}/\mathfrak{h}_{v}) + \sum_{e \in E}\left( \dim(\mathfrak{g}/\mathfrak{h}_{e}) - \dim(\sum_{v\in e} \mathfrak{g}/\mathfrak{h}_{v}) \right).
\end{align*}
We note that this implies \Cref{thm: weak_maxwell}, since $\dim(H^{1}(I(\Gamma), \mathcal{M})) \geq 0$. In addition, one has that the long exact sequence $\eqref{les}$ carries essential information about rigidity. If $\Gamma$ is connected, one has $H^{0}(\Gamma, \underline{\mathfrak{g}}) = \mathfrak{g}$ and the map $H^{0}(I(\Gamma), \underline{\mathfrak{g}}) \rightarrow H^{0}(I(\Gamma), \mathcal{M})$ is the map $i$ from \Cref{lem: Trivial_motions}, and $\ker(i) \cong H^{0}(I(\Gamma), \mathcal{S})$. One has that \Cref{thm: Necessary condition} gives, that if $H^{0}(I(\Gamma), \mathcal{M}) = \mathfrak{g}$ and $H^{1}(I(\Gamma), \mathcal{M}) = 0$, then for any sub hypergraph $(V', E')$ with $\bigcap_{v \in V'} \mathfrak{h}_{v} = 0$, one has 
\begin{equation*}
    \lambda |E'| \leq  d |V'| -\dim(G).
\end{equation*}
for the parameters
\begin{align*}
    \lambda &= \sum_{v\in e}\dim(\mathfrak{g}/\mathfrak{h}_{v}) - \dim(\mathfrak{g}/\mathfrak{h}_{e})\\
    d &= \dim(\mathfrak{g}/\mathfrak{h}_{v}),
\end{align*}
If $\bigcap_{v\in V} \mathfrak{h}_{v}=0$, one has for the entire hypergraph:
\begin{equation*}
    \lambda |E| = d|V| -\dim(G).
\end{equation*}
This implies the sparsity condition from \cite{graphsofgroups}.

\begin{example}\label{ex: Se(d) Maxwell}
    Let $\rho$ be a graph-of-groups realisation as in \Cref{ex: se(d)}, such that $\rho(v)\neq \rho(w)$ for all edges $vw\in E$. One then gets
    \begin{align*}
        \dim(H^{0}(I(\Gamma), \mathcal{M})) - \dim(H^{1}(I(\Gamma), \mathcal{M})) & = d |V| +(2d-1)|E|- 2d |E|.\\
        &= d|V| -|E|
    \end{align*}
If $H^{1}(I(\Gamma), \mathcal{M})= 0$, then one has the following sparsity condition, under mild geometric conditions. If $k$ points $p_1, \dots, p_k$, with stabilisers $\mathfrak{h}_1,\dots \mathfrak{h}_k$ span a subspace of dimension at least $d-1$, then $\bigcap_{i=1}^{k} \mathfrak{h}_{p_i}= 0$. So, if $\rho$ comes from a bar joint framework such that the points $\{p(v)~\vert~v\in V\}$ are in general position, then for any subset with more than $d$ points one has 
        \begin{align*}
        |E'| &\leq d|V'| -\binom{d+1}{2}.
        \end{align*}
        Moreover, if $|V|\geq d$ and if $p$ is infinitesimally rigid, one has 
        \begin{align*}
        |E| &= d|V| -\binom{d+1}{2}.
        \end{align*}
    This is of course the classical sparsity condition. 
\end{example}

\subsection{Genericity of infinitesimal rigidity for graph-of-groups realisations.}\label{sec: genericity for graph of groups realisations}
The goal of this section is to show that the motion sheaves arising from graph-of-groups realisations can be seen as smooth semi-algebraic subsets of $M_{s, n}$. Applying the semi-continuity result from \Cref{thm: Genericity_generalised_frameworks.} will then imply that their rigidity and independence are generic properties. For clarity, when given an algebraic group $G$ defined over $\mathbb{R}$, we shall use $G(\mathbb{R})$ to denote the (Lie) group of real points. Given an algebraic group $G$ defined over $\mathbb{R}$, by an algebraic group action of $G(\mathbb{R})$ on a real algebraic variety $X$, we mean an action $G(\mathbb{R}) \times X \rightarrow X$ that is given by regular functions (i.e. by polynomials).

\begin{proposition}\label{dense-Euclidean}
    Let $X \subseteq \mathbb{R}^{n}$ be an irreducible semi-algebraic set which is also a smooth manifold. Then, if $U\subseteq \mathbb{R}^{n}$ is a Zariski open set, then either $U\cap X$ is empty, or $U\cap X$ is dense and open for the Euclidean topology on $X$.  
\end{proposition}
\begin{proof}
Let $V=X\setminus U$, which is defined by some polynomial $p$, i.e. $V = \{ x\in X ~\vert~p(x) = 0\}$ \cite[Proposition 2.1.3.]{Bochnak1998}. We see that $V$ is a Zariski closed subset of $X$, and hence by irreducibility, either $V=X$ (in which case $U\cap X= \emptyset$), or $V$ is semi-algebraic subset of dimension $ k < \dim(X)$. In the latter case, since $X$ is a smooth manifold, for every $x\in V$ and every open ball $O\subseteq X$ (for the Euclidean topology) which contains $x$, $O\cap V\neq O$. Indeed, otherwise the local dimension satisfies $\dim(V_x) = \dim(X)$ by \cite[Proposition 2.8.14]{Bochnak1998}, which is a contradiction since $\dim(V_x)\leq k < \dim(X)$.
\end{proof}

\begin{lemma}\label{lem: action-V-action-gr(k,V)}
    Let $G$ be an algebraic group defined over $\mathbb{R}$. Suppose that $\psi: G(\mathbb{R})\rightarrow \textup{Gl}(\mathbb{V})$ is a representation of a real algebraic group, i.e. the map $G(\mathbb{R})\times \mathbb{V} \rightarrow \mathbb{V}$ is regular and $\psi(g)$ is a linear map for each $g\in G$. Then $G(\mathbb{R})$ acts algebraically on $\textup{Gr}(s, \mathbb{V})$.
\end{lemma}
\begin{proof}
One can use Plücker coordinates and apply the same reasoning as in \cite[Theorem 10.19]{harris2013algebraic}.
\end{proof}

If we take $G$ to be a Zariski-closed subgroup of $\textup{Gl}(n, \mathbb{R})$, one has that the adjoint action of $G(\mathbb{R})$ on $\mathfrak{g} \leq \mathbb{R}^{n^2}$ is simply given by
\begin{equation*}
    \textup{Ad}(g)\cdot X =g X g^{-1}.
\end{equation*}
This action is clearly defined using polynomial functions, and hence we get an algebraic action on $\textup{Gr}(s, \mathfrak{g})$. The orbit of any $\mathfrak{h} \leq \mathfrak{g}$ is thus a semi-algebraic subset of 
$\textup{Gr}(s, \mathfrak{g})$. We can describe classes of motion sheaves arising from graph-of-groups realisations in this way.

\begin{definition}
    Let $G$ be a Lie group and let $H$ be a closed Lie subgroup. Given a hypergraph $\Gamma$, we define $C_{G, H}$ to be the class of graph-of-groups realisations such that for every $v\in V$ and $e\in E$, one has:
    \begin{align*}
        \rho(v) & = g_{v} H g_{v}^{-1} \textup{ for some } g_{v} \in G,\\
        \rho(e) & = \bigcap_{v\in e} \rho(v).
    \end{align*}
    We can identify $C_{G, H}$ with $(G/N_{G}(H))^{V}$, where $N_{G}(H) = \{g \in G~\vert~ gHg^{-1} = H\}$, which gives $C_{G, H}$ the structure of a smooth manifold.
\end{definition}

Suppose that we are given $\rho \in C_{G, H}$. Then the motion sheaf associated to $\rho$ is defined using a sheaf $\mathcal{S}$ such that for every $v$ one has 
$\mathcal{S}(v) = \textup{Ad}(g_{v})\mathfrak{h},$
where $\mathfrak{h}$ is the Lie algebra of $H$. We can hence describe the set of motion sheaves as those motion sheaves belonging to $(G \cdot  \mathfrak{h})^{V} \subseteq \textup{Gr}( s , \mathfrak{g})$, where $G \cdot \mathfrak{h}$ is the orbit of $\mathfrak{h}$ under the adjoint action and $s= \dim(\mathfrak{h})$. Letting $N_{G}(\mathfrak{h}) = \{g\in G~\vert~ \textup{Ad}(g)(\mathfrak{h}) = \mathfrak{h} \}$, one can equivalently describe this set as the image of the injective natural map
\begin{align*}
    (G/N_{G}(\mathfrak{h}))^{V} \rightarrow \textup{Gr}(s, \mathfrak{g})^{|V|}: (g_{v}N_{G}(\mathfrak{h}))_{v\in V} \mapsto  (\textup{Ad}(g_{v})(\mathfrak{h}))_{v \in V}.
\end{align*}

In the case that $H$ is connected, the following lemma shows that one can replace $N_{G}(\mathfrak{h})$ by $N_{G}(H)$.
\begin{proposition}\label{lem: lie-alg-lie-group-norm}
    Suppose that $G$ is a Lie group, and $H$ is a connected closed Lie subgroup. Then $N_G(\mathfrak{h}) = N_G(H)$.
\end{proposition}
\begin{proof}
Since $H$ is connected, any element of $H$ can be written as a product $e^{X_1},\dots, e^{X_n}$, where $X\mapsto e^{X}$ denotes the exponential map. Thus, if for elements $X_i \in \mathfrak{h}$ one has $ge^{X_i}g^{-1}=e^{X_i'}$ for some $X_i'\in \mathfrak{h}$, then $gHg^{-1} = H$. Since $ge^{X_i}g^{-1} = e^{(\textup{Ad}(g)(X_i))}$ (by naturality of the exponential map), it follows that $N_G(\mathfrak{h}) \subseteq N_G(H)$. Since the reverse inclusion always holds, this completes the proof.
\end{proof}

\begin{theorem}\label{thm: Generic}
    Let $\Gamma$ be a hypergraph. Let $G\subseteq Gl_n$ be a connected linear algebraic group defined over $\mathbb{R}$, and let $H\subseteq G(\mathbb{R})$ be a closed and connected subgroup (for the Euclidean topology). Then, there is a dense open subset $U \subseteq C_{G(\mathbb{R}),H}$ (for the Euclidean topology) such that for all $\rho\in U$, the associated motion sheaves $h^{0}(I(\Gamma), \mathcal{M})$ and $h^{1}(I(\Gamma), \mathcal{M})$ have the same constant dimension on $U$. 
\end{theorem}
\begin{proof}
     Let $X = G(\mathbb{R}) \cdot \mathfrak{h}$ be the semi-algebraic subset of $\textup{Gr}(s,\mathfrak{g})$. One sees that $X$ is irreducible, being the image of an irreducible subset under the regular map $G(\mathbb{R})\rightarrow \textup{Gr}(s,\mathfrak{g}): g \mapsto \textup{Ad}(g)\mathfrak{h}$, and $X^{|V|}$ is irreducible as well \cite[Theorem 2.8.3.(iii)]{Bochnak1998}. 
    By \Cref{lem: Intersections_of_subspaces}, we see that there is a Zariski open subset of $O\subseteq X^{|V|}$ such that $\dim(\bigcap_{v\in e} A_v)$ is minimal for all $e\in E$ and $(A_1, \dots, A_{|V|})\in O$. By \Cref{thm: Genericity_generalised_frameworks.}, we see that there is a nonempty Zariski open subset $U$ of $O$, such that $h^{0}(I(\Gamma), \mathcal{M})$ and $h^{1}(I(\Gamma), \mathcal{M})$ are minimal on $U$. 
    
    Since $H$ is connected, by \Cref{lem: lie-alg-lie-group-norm} and by the discussion above,  the map $$\varphi: (G(\mathbb{R})/N_{G(\mathbb{R})}(H))^{|V|} \rightarrow \textup{Gr}(s, \mathfrak{g})^{|V|}: (g_{v}N_{G(\mathbb{R})}(H))_{v\in V} \rightarrow (\textup{Ad}(g_{v})(\mathfrak{h}))_{v\in V}$$
    is a bijective map onto $X^{|V|}$, and by generalities on orbits, it is an immersion (see \cite[Theorem 3.29]{Kirillov}). Note that there exists an $x\in X^{|V|}$ and some $W\subseteq M_{s,n}(\Gamma)$ which is open for the Euclidean topology such that $x\in W$ and $W\cap X^{|V|}$ is a smooth manifold (this follows since $X^{|V|}$ admits a Nash stratification, \cite[Definition 9.1.7, Proposition 9.1.8]{Bochnak1998}). Then, since for any $y\in X^{|V|}$, there exists a $g\in G^{|V|}$ such that $y = g \cdot x$, we see that $X^{|V|} \cap g W$ is a manifold. Since this holds for arbitrary $y$, we see that $X^{|V|}$ is an embedded submanifold. 
    
    It follows that $\varphi$ is an embedding. The set $U\subseteq X^{|V|}$ is dense and open with respect to the Euclidean topology by Proposition \ref{dense-Euclidean}, since $X^{|V|}$ is smooth and irreducible. Hence, since $\varphi$ is an embedding, $\varphi^{-1}(U)$ satisfies the conclusion of the theorem. 
\end{proof}

The above theorem may be reformulated as rigidity and independence being generic properties for motion sheaves coming from graphs of groups, when $G$ is connected.
In the case where $G$ is a Lie group, $\dim(H)=1$, and if $N_{G}(H)/H$ is finite, the closure of the image $\varphi(G/H)$ is particularly well-behaved. 

\begin{theorem}\label{cor: main - thm}
    Let $\Gamma$ be a graph. Let $G$ be a Lie group, and suppose that $H\subseteq G$ is a $1$-dimensional connected subgroup such that $N_{G}(H)/H$ is finite. Then, for any graph $\Gamma$, there is a dense open subset of $\mathcal{M} \in C_{G, H}^{|V|}$ such that $h^{1}(I(\Gamma), \mathcal{M}) = 0$ if and only if $(\dim(G) -2)\cdot \Gamma$ is $(\dim(G) - 1, \dim(G))$ sparse.
\end{theorem}
\begin{proof}
    We have that $\textup{Gr}(1, \mathfrak{g})$ is a projective space, which we denote by $P(\mathfrak{g})$. The map $G/N_{G}(H) \rightarrow \mathbb{P}(\mathfrak{g})$ is an immersion, and since $\dim( G/N_{G}(H)) = \dim(P(\mathfrak{g}))$, it is an embedding \cite[Proposition 5.21(a)]{Lee}. Thus, $\varphi: C_{G, H}^{|V|}\rightarrow \mathbb{P}(\mathfrak{g})^{|V|}$ is also an embedding, and $\varphi(C_{G,H}^{|V|})$ is hence an open subset for the Euclidean topology on $\textup{Gr}(1, \mathfrak{g})$. By Theorem \ref{thm: main theorem} and Proposition \ref{dense-Euclidean}, since $P(\mathfrak{g})$ is smooth, we see that for every graph $\Gamma$, there is a dense open subset $U\subseteq P(\mathfrak{g})^{|V|}$ such that $H^{1}(I(\Gamma), \mathcal{M}) = 0$ if and only if $(n-2)\cdot \Gamma$ is $(n-1, n)$-tight. Now, since it is easy to show that for any open set $O\subseteq P(\mathfrak{g})$, one has that $U\cap O$ is dense and open in $O$, we see that $\varphi^{-1}(U)$ is open in $C_{G, H}^{V}$. 
\end{proof}

One can thus take \textit{any Lie group $G$}, and \textit{any} $1$-dimensional connected subgroup $H \subseteq G(\mathbb{R})$ such that $N_{G(\mathbb{R})}(H)/H$ is finite and get a characterisation. If $H$ is not connected, taking the connected component $H_0$ of $H$ leads to the same notions of local rigidity and infinitesimal rigidity, so even when $H$ is not connected, one gets a characterisation (granted that $N_G(H_0)/H_0$ is finite).

If one takes $E(2)$ and $O(2)$ as in \Cref{ex: se(d)}, one recovers the Geiringer-Laman theorem as an instance of \Cref{cor: main - thm}. Using the groups $PSL_{2}$ and $O(3)$ instead, it follows that the rigid graphs on the hyperbolic plane and the sphere are characterised by the same sparsity condition as for the Euclidean plane. 

Taking $1$-dimensional subgroups of $E(d)$, one sees that the resulting graph-of-groups realisations are minimally rigid if $\left(\binom{d+1}{2} - 2\right)\cdot \Gamma$ is $\left(\binom{d+1}{2} - 1, \binom{d+1}{2}\right)$ tight. Geometrically, this corresponds to each vertex representing an $n-2$ dimensional subspace (for example, a line in $3$-space), and if there is an edge between two such $n-2$ subspaces, then one is only allowed to rotate around the other. Such a framework can be interpreted as a special case of the $(n-2, 2)$ frameworks in \cite{Tay1989}, also known as panel-bar frameworks, where there are either exactly $\binom{n+1}{2}- 2$ or $0$ edges between each pair of $n-2$ dimensional subspaces. Thus, in this case, we recover a special case of Tay's result. Perhaps by interpreting the associated sheaves in a suitable way, one can fully recover Tay's result as a special case of \Cref{thm: Genericity_F}.

We end this section by describing the semi-algebraic sets parametrising the motion sheaves coming from Euclidean rigidity when $d=2$ and $d=3$. The groups $SE(d)$ can be considered as algebraic groups. We will simply write $SE(d)$ instead of $SE(d)(\mathbb{R})$.

\begin{example}
We consider $G=SE(2)$ and $H=SO(2)$, where $H$ is seen as the stabiliser of the point $(0,0)$. We let $\mathfrak{o}_2(\mathbb{R})\subseteq \mathfrak{e}_2(\mathbb{R})$ be the Lie algebra corresponding to the $H$, which is then given by the matrices
\begin{equation*}
    \begin{bmatrix}   0 & -\omega & 0\\
        \omega & 0 & 0\\
        0 & 0 & 0
    \end{bmatrix},
\end{equation*}
with $\omega\in \mathbb{R}$. We note that $\mathfrak{o}_{2}(\mathbb{R})$ defines a point of $\textup{Gr}(1,\mathfrak{e}_2(\mathbb{R})) = \mathbb{R}P^{2}$. One can explicitly verify that if $$g=
        \begin{bmatrix}
        \cos(\theta) & - \sin(\theta) & a\\
        \sin(\theta) & \cos(\theta) &b\\
        0 & 0 & 1
    \end{bmatrix},
$$
then
\begin{align*}
    \textup{Ad}(g)(\mathfrak{o}_2(\mathbb{R})) &=
\{ \begin{bmatrix}
        0 & -\omega & \omega b \\
        \omega & 0 & -\omega a \\
        0 & 0 & 0
    \end{bmatrix}~ \in \mathfrak{e}_{2}(\mathbb{R})~ \vert~ \omega \in \mathbb{R}\}.
\end{align*}
One way to compute this is to pick the specific element $$g'= \begin{bmatrix}
    1 & 0 & a\\
    0& 1  & b\\
 0 & 0 & 1
\end{bmatrix}$$
that represents the coset $g\cdot O(2)$ and calculate $Ad(g')\mathfrak{o}_2(\mathbb{R}) = g'\mathfrak{o}_2(\mathbb{R})(g')^{-1}$ for this specific $g'$.
In any case, when $$g\in \begin{bmatrix}
    1 & 0 & a\\
    0& 1  & b\\
 0 & 0 & 1
\end{bmatrix}O(2),$$ one sees that $g\mathfrak{o}_{2}g^{-1}$ is the stabiliser of the point $(a,b) \in \mathbb{R}^2$. Thus, the image of the map $\varphi: G/H\rightarrow \mathbb{R}P^2$ can be seen as an affine open subset of $\mathbb{R}P^2$, and its closure is $\mathbb{R}P^2$. 
\end{example}

\begin{example}
The example of $3$-dimensional rigidity is more complicated. We let $G=SE(3)$, and we let $H=SO(3)$, again viewed as the stabiliser of the identity in $\mathbb{R}^{3}$. We see that for $$g = \begin{bmatrix}   1 & 0 & 0 & a\\
        0 & 1 & 0 & b\\
        0 & 0 & 1 & c \\
        0 & 0 & 0 & 1
    \end{bmatrix},$$

the algebra $\textup{Ad}(g)\mathfrak{o}(3)$ is given by
\begin{align*}
    \textup{Ad}(g)(\mathfrak{o}_3(\mathbb{R})) &=
    \begin{bmatrix}   0 & -\omega_3 & \omega_2 & \omega_3 b -\omega_2 c\\
        \omega_3 &0  & - \omega_1 & - \omega_3 a +\omega_1 c\\
        - \omega_2& \omega_1 & 0 & \omega_2 a - \omega_1 b \\
        0 & 0 & 0 & 0
    \end{bmatrix},
    \end{align*}
    where $\omega_1, \omega_2, \omega_3\in \mathbb{R}$, which is the stabiliser of the points $(a,b,c)\in \mathbb{R}^{3}$. Let us now consider what the image of the map $\varphi: G/H \rightarrow Gr(3, \mathfrak{e}(3))$ is.  It is a $3$-dimensional linear subspace inside of this Grassmannian (which is $9$ dimensional). Take the coordinates $\omega_1, \omega_2, \omega_3, a, b, c$, and the standard basis vectors for these coordinates. Then, any stabiliser can be described with a basis given by the rows of the following matrix
    \begin{align*}
    \begin{bmatrix}   
    1 & 0 & 0 & 0 & z & - y\\
    0 & 1 & 0 & -z & 0 & x\\
    0 & 0 & 1 & y & -x & 0\\
    \end{bmatrix}.
    \end{align*}
    Using coordinates on the Grassmannian given by:
    \begin{align*}
    \begin{bmatrix}   
    1 & 0 & 0 & a_{1,1} & a_{1,2} & a_{1,3}\\
    0 & 1 & 0 & a_{2,1} & a_{2,2} & a_{2,3}\\
    0 & 0 & 1 & a_{3,1} & a_{3,2} & a_{3,3}\\
    \end{bmatrix},
    \end{align*}
    we see that $\varphi(G/H)$ is cut out by the equations 
    \begin{equation*}
    \begin{array}{clclcl}
    a_{1,1} &=0 & a_{2,2} &=0 & a_{3,3} &=0 \\
    a_{1,2} &=-a_{2,1} & a_{1,3} &=-a_{3,1} & a_{2,3} &=-a_{3,2}\\
    \end{array}
    \end{equation*}
    Thus, we see that with these coordinates $\overline{\varphi(G/H)}$ is the intersection of $\textup{Gr}(3,6)$ with a linear space. The resulting subvariety is, however, quite special. Indeed, one has that for any two subspaces $A_1\in \overline{\varphi(G/H)}$ and $A_2\in \overline{\varphi(G/H)},$ one has $\dim(A_1\cap A_2) \geq 1$, while for generic $A_1, A_2\in \textup{Gr}(3,6),$ one has $\dim(A_1\cap A_2) =0$.
\end{example}

\section{Examples}
\subsection{Inductive constructions for Euclidean rigidity}\label{sec: inductive constructions gogs}
We make some remarks about how  \Cref{thm: k-extension} applies to graph-of-groups realisations coming from bar-joint frameworks. In \cite{graphsofgroups}, it was noticed that $3$ distinct points $x,y,z\in \mathbb{R}^{2}$ are collinear if and only if $\mathfrak{h}_x + \mathfrak{h}_y + \mathfrak{h}_z$ has dimension $2$, where $\mathfrak{h}_x$ is the Lie algebra of the stabiliser of $x$. This generalises this to arbitrary dimension. 

\begin{proposition}\label{prop: collinear}
Let $d\geq 2$, and let $x,y,z\in \mathbb{R}^{d}$ be distinct points. Let $\mathfrak{h}_x, \mathfrak{h}_y, \mathfrak{h}_z \subseteq \mathfrak{e}(d)$ be the stabilisers of these points. The points $x,y,z$ are collinear if and only if $\mathfrak{h}_x \subseteq \mathfrak{h}_y + \mathfrak{h}_z$
\end{proposition}
\begin{proof}
    Suppose $x,y,z$ are collinear. Then one can write $x= \lambda_1 y + \lambda_2 z$, where $\lambda_1 + \lambda_2 =1$. Let $\Omega$ be a skew-symmetric matrix. One then has
    \begin{equation*}
     \lambda_1
        \begin{bmatrix}
           \Omega  &  - \Omega \cdot y\\
            0 & 0
        \end{bmatrix}+  \lambda_2
        \begin{bmatrix}
           \Omega  &  - \Omega \cdot z\\
            0 & 0
        \end{bmatrix}
        =         \begin{bmatrix}
           \Omega  &  - \Omega \cdot x\\
            0 & 0
        \end{bmatrix},
    \end{equation*}
    which shows clearly that $\mathfrak{h}_x \subseteq \mathfrak{h}_y +\mathfrak{h}_z$. This shows 'only if'. For 'if', we first note that without loss of generality we may take $y=0$ and $z=(1, 0, \dots, 0).$ We see that elements of $\mathfrak{h}_y +\mathfrak{h}_z$ are of the form
    \begin{align*}
     \begin{bmatrix}
           \Omega  & 0\\
            0 & 0
        \end{bmatrix}+  \lambda_2
        \begin{bmatrix}
           \Omega'  &  - \Omega' \cdot z\\
            0 & 0
        \end{bmatrix}
    \end{align*}
    The part $- \Omega' \cdot z$ will be of the form $[0, a_1, \dots, a_{n}]$. Now for any $x\in \mathbb{R}^{d}$, such that $x =(x_1, \dots, x_i, \dots, x_d)$, with at least some $x_i \neq 0$ for $i\neq 1$, one defines the skew-symmetric matrix $\Omega_i$
    \begin{equation*}
    (\Omega_i)_{k, \ell} =\begin{cases} 1
        & \textup{ if } k=1, \ell = i\\
        - 1
        & \textup{ if } k= i, \ell = 1\\
        0
        & \textup{ otherwise }
    \end{cases}.
    \end{equation*}
    It is then clear that $\begin{bmatrix}
        \Omega_i & - \Omega_i x\\
        0 & 0
    \end{bmatrix}\notin \mathfrak{h}_y +\mathfrak{h}_z$, but it is clearly an element of $\mathfrak{h}_x$, completing the proof.
\end{proof}

Using \Cref{prop: collinear} and \Cref{thm: k-extension}, it is easy to show that the $0$ and $1$-extension works in arbitrary dimensions, and the construction basically follows the usual argument. From \Cref{prop: collinear}, it also follows that one cannot apply  \Cref{thm: k-extension} to understand the $2$-extension in $3$ dimensions, unless the $4$ points where one performs the extension happen to lie on a plane.

\subsection{Parallel Redrawings}

In this final section, we show that the class of motion sheaves arising from parallel-redrawings problems leads to a very well-behaved class of problems inside of $M_{s,n}(\Gamma)$.

Let $G \leq \textup{Aff}(n, \mathbb{R})$ be the group generated by dilations and translations. Using homogeneous matrices, group elements $g\in G$ are represented by 
\begin{equation*}
    \begin{bmatrix}
        \lambda & 0 & \dots & 0 & x_1 \\
        0 &   \lambda & \dots & 0 & x_2 \\
        \vdots   &   \vdots  & \ddots & \vdots & \vdots \\
        0 &   0 & \dots & \lambda & x_n \\
        0 &   0 & \dots & 0 & 1 \\
    \end{bmatrix} \textup{ where } \lambda\in \mathbb{R}^{*}, x_1, \dots, x_n \in \mathbb{R}
\end{equation*}
The Lie algebra $\mathfrak{g}$ of this group consists of $(n+1) \times (n+1)$ matrices
\begin{equation*}
    \begin{bmatrix}
        t & 0 & \dots & 0 & x_1 \\
        0 &   t & \dots & 0 & x_2 \\
        \vdots   &   \vdots  & \ddots & \vdots & \vdots \\
        0 &   0 & \dots & t & x_n \\
        0 &   0 & \dots & 0 & 0 \\
    \end{bmatrix} \textup{ where } t, x_1, \dots, x_n \in \mathbb{R}
\end{equation*}

We will do computations in this algebra using the basis 
\begin{equation*}
\begin{array}{clcl}
L &=   \begin{bmatrix}
        1 &\dots &  0 & 0\\
         \vdots & \ddots & \vdots & \vdots\\
        0 &\dots &  1 & 0\\
        0 &\dots &  0 & 0\\
\end{bmatrix}, &
    T_i  &= \begin{bmatrix}
        0 &\dots &  0 & 0 \\
         \vdots & \ddots & \vdots & \vdots\\
        0 &\dots &  0 & 0 \\
        0 &\dots &  0 & 1\\
        0 &\dots &  0 & 0 \\
         \vdots & \ddots & \vdots & \vdots\\
                     0 &\dots &  0 & 0 \\
    \end{bmatrix},
\end{array}
\end{equation*}
where the nonzero row in $T_i$ is the $i$-th row, and $i\in \{1, \dots n\}$.  

\begin{lemma}\label{lem: Parallel-subalgebras}
    Any vector-subspace of $\mathfrak{g}$ is a subalgebra.
\end{lemma}
\begin{proof}
    We compute using the basis $L, T_1, \dots T_n$ defined above. We remark that
    \begin{align*}
        [L, T_i] &= T_i \textup{ for all } i\in \{1, \dots, n\}\\
        [T_i, T_j] &= 0 \textup{ for all } i,j \in \{1, \dots, n\}.
    \end{align*}
    Now, let $A\leq \mathfrak{g}$ be a vector-subspace of $\mathfrak{g}$. If $A\leq \textup{Span}(T_1, \dots, T_n)$, $[X,Y] = 0$ for all $X,Y\in A$. We may thus write $A= \textup{span}(X_1, \dots X_{s})$, and without loss of generality
    \begin{align*}
        X_1 &= L + \sum_{i=1}^{n} a_{1i} T_i\\
        X_j &= \sum_{i=1}^{n} a_{ji} T_i \text{ for } j\in \{2, \dots s\}
    \end{align*}
    Now, since
    \begin{align*}
        [X_1, X_j]& = X_j \textup{ for all }j\in \{2, \dots, s \}\\
        [X_i, X_j]&=0\textup{ for all } i,j \in \{2, \dots, s \},
    \end{align*}
    the conclusion clearly holds. 
\end{proof}

\begin{lemma}\label{lem: subalgebra of G is stabiliser}
    Any subalgebra $A$ that contains an infinitesimal generator of a dilation, given by $D= L + \sum_{i=1}^{n} a_iT_i $, is the Lie algebra of the stabiliser of an affine subspace of $L_A \subseteq \mathbb{R}^{n}$ of dimension $\dim(A) - 1$. Moreover, for such subspaces $L_{A + B} = L_{A} \vee L_{B}$, where $L_A \vee L_B$ is the smallest affine subspace containing both $L_A$ and $L_B$ and $L_{A\cap B} = L_A \cap L_B$, whenever $A\cap B$ contains a dilation.
\end{lemma}
\begin{proof}
We can choose a basis containing $D$ and infinitesimal generators of translations $\tau_{k} = \sum_{i=1}^{n} a_{ki} T_i,$ for $k\in \{1, \dots s\}$. Letting $P_D = (-a_1, \dots, -a_n) \in \mathbb{R}^n$, one has that $D\cdot P_D = 0$, so the element $D$ acts trivially on $P_D$ under the infinitesimal action of $\mathfrak{g}$ on $\mathbb{R}^{n}$. Let 
\begin{align*}
    v_1 &= (a_{11}, \dots, a_{1n})\\
     & \vdots\\
    v_s &= (a_{s1}, \dots, a_{sn}).
\end{align*}
We then let $L_{A}:= P_D + \textup{Span}(v_1, \dots v_s)$. This is an affine subspace of dimension $s = \dim(A) - 1$. We additionally define $v_0 = (a_1, \dots a_n)$.
Now, one notes that $A$ is the Lie algebra of the subgroup generated by the matrices 
\begin{align*}
\begin{bmatrix}
\lambda \textup{Id}_n & (\lambda -1)v_0^{T} + t_1 v_1^{T} + \dots + t_s v_s^{T}\\
 0 &  1 
 \end{bmatrix},
\textup{ where } \lambda \in \mathbb{R}^{*}, t_1, \dots, t_s \in \mathbb{R}
\end{align*}
This subgroup is precisely the stabiliser of $L_A$, and we note that it consists precisely of the dilations fixing a point of $L_A$, and translations in the direction of $L_A$ (i.e., translations with vectors in $\textup{Span}(v_1, \dots v_s)$). To finish the proof, we prove that $L_{A\cap B} = L_{A} \cap L_{B}$ and $L_{A+B} = L_A \vee L_B$. We first note that by the description of the stabilisers above, one has that $L_A \subseteq X$ if and only if $A$ fixes $X$. Hence, $A\subseteq B$, if and only if one has $L_A \subseteq L_B$. The claim about the sum and intersection then follows easily. 
\end{proof}

\begin{remark}
    By associating to subalgebras containing only translations, the space of points at infinity giving the directions of the translations, one obtains using the construction in \Cref{lem: subalgebra of G is stabiliser}, a lattice isomorphism between subspaces of $\mathbb{P}^{n}$ and subalgebras of $\mathfrak{g}$. 
\end{remark}

We can use these observations to describe various 'parallel redrawings' type problems using graphs of groups.

\begin{definition}
    Let $\Gamma=(V,E)$ be an $r$-regular hypergraph. A realisation of $\Gamma$ as a subspace arrangement of $\Gamma$ with vertex dimensions $s$, is an assignment $$r: V \rightarrow \textup{Gr}(s + 1, \mathbb{R}^{n+1}): v\mapsto r(v),$$ such that $r(v)$ is an $s$-dimensional affine subspace of $\mathbb{R}P^{n}$ for each $v\in V$. One defines $r(e) := \bigvee_{v\in e} L_v$ for $e\in E$. One says that two realisations $r, r'$ are parallel redrawings of each other if $r(e)$ is parallel to $r'(e)$ for every $e\in E$.
\end{definition}

To any vertex-regular subspace arrangement $r$, we assign a Lie graph-of-groups realisation $\rho_{r}(\Gamma)$ by setting
\begin{align*}
    &v\mapsto \textup{Stab}(r(v))\\    
    &e \mapsto \cap_{v\in e}\textup{Stab}(r(v)).
\end{align*}
One can show, in a completely analogous way as in \cite[Theorem 4.3]{graphsofgroups}, that there is a $1-1$ correspondence between subspace arrangements $r$ of vertex dimension $s$ and graph-of-groups realisations given by $\rho_r$ and $r$ is a parallel redrawing of $r'$ if and only if there exists a motion $M$ of $\rho_r$ such that $M(\rho_r) = \rho_{r'}$.

\begin{lemma}\label{lem : Parallel lemma}
    The set of motion sheaves $\mathcal{M} \in \textup{Sh}(I(\Gamma))$ that come from a graph-of-groups realisation $\rho_{r}$, for $r$ a subspace arrangement of vertex dimension $s$ is dense in $\textup{Gr}(s +1, \mathfrak{g})$. 
\end{lemma}
\begin{proof}
    We again use the basis $L, T_1, \dots, T_n$ for $\mathfrak{g}$. We note that a subspace $A\in \textup{Gr}(s+1, \mathfrak{g})$ is a stabiliser of an affine subspace if and only if $A \cap \langle T_1, \dots T_n\rangle \neq A$ by \Cref{lem: subalgebra of G is stabiliser} and \Cref{lem: Parallel-subalgebras}. Since this defines a Zariski open subset of  $\textup{Gr}(s+1, \mathfrak{g})$, the lemma follows. 
\end{proof}

\begin{corollary}\label{parallel-graphs}
    Let $\Gamma$ be a graph, and let $n\in \mathbb{N}$. For almost all realisations of $\Gamma$ as a subspace arrangement in $\mathbb{R}^{n}$ with vertex dimension $0$, one has that the associated motion sheaf satisfies $h^{1}(\mathcal{M}) = 0$ if and only if $(n-2)\Gamma$ is $(n, n+1)$ sparse. 
\end{corollary}
\begin{proof}
    Follows from \Cref{thm: main theorem} and \Cref{lem : Parallel lemma}. 
\end{proof}

\Cref{parallel-graphs} yields a characterisation for so-called parallel rigidity, and it had already been proven (see \cite[Theorem 8.2.2]{Whiteley1996}, \cite[Theorem 6.5]{develin2007rigidity}). If \Cref{Szegö conjecture} is true, then for $s\leq \frac{n+2}{4}$, one could prove in the same way that $h^{1}(\Gamma, \mathcal{M}) = 0$ holds for generic $\mathcal{M}$ coming from a subspace arrangement of vertex dimension $s$ if and only if $(n-2s) \Gamma$ is $(n-s, n)$ sparse. We remark that \Cref{parallel-graphs} also follows from \Cref{cor: main - thm}.  

In fact, the so-called problem of parallel redrawings/liftings of scenes \cite{Whiteley1989} also has a characterisation in terms of a sparsity condition of the hypergraph. This problem on hypergraphs can also be described in terms of graph-of-groups realisations (see \cite[Section 4.3]{graphsofgroups}). The fact that the resulting motion sheaves are dense in $M_{s, n}(\Gamma)$ suggests that a more general theorem may hold for $M_{s,n}(\Gamma)$, and that the underlying reason for the characterisation is that the resulting motion sheaves are dense in $M_{s, n}(\Gamma)$.

\section{Acknowledgements}

This work was partially supported by the Wallenberg AI, Autonomous Systems and
Software Program (WASP) funded by the Knut and Alice Wallenberg Foundation. I would also like to thank Klara Stokes for useful discussions and comments. 

\bibliography{refs}

\end{document}